\documentclass[a4paper,12pt]{article}

\usepackage{pstricks}

\usepackage{amsmath}
\usepackage{amsfonts}
\usepackage{amssymb}
\usepackage{amsopn}
\usepackage{amscd}

\usepackage[all]{xy}

\usepackage{graphicx}

\usepackage[T1]{fontenc}
\usepackage[latin1]{inputenc}
\usepackage{url}

\usepackage{color}
\definecolor{mygray}{gray}{.25}

\usepackage{bm}

\DeclareMathAlphabet{\mathpzc}{OT1}{pzc}{m}{it}

\DeclareMathOperator{\Id}{Id}
\DeclareMathOperator{\Spec}{Spec}
\DeclareMathOperator{\Gl}{GL}

\DeclareMathOperator{\Hom}{Hom}

\DeclareMathOperator{\Ker}{Ker}

\DeclareMathOperator{\Dim}{dim}
\DeclareMathOperator{\Deg}{deg}

\DeclareMathOperator{\init}{in}
\DeclareMathOperator{\Fitt}{Fitt}

\DeclareMathOperator{\state}{state}
\DeclareMathOperator{\dets}{dets}

\usepackage{amsthm}

\usepackage[plainpages=false,pdfpagelabels]{hyperref}
\hypersetup{
    pdftoolbar=true,        																		 			
    pdfmenubar=true,        																		 			
    pdffitwindow=false,     																		 			
    pdfstartview={FitH},    																		 			
    pdftitle={Non-coherent Components of the Toric Hilbert Scheme},   
    pdfauthor={René Birkner},     															 			
    pdfsubject={Subject},   																		 			
    pdfcreator={René Birkner},   																 			
    pdfproducer={René Birkner}, 																 			
    pdfkeywords={toric Hilbert scheme, 
    						 A-graded, 
    						 non-coherent, 
    						 algebraic geometry, 
    						 toric geometry}, 													 		 			
    pdfnewwindow=true,      																		 			
    pdfdisplaydoctitle=true,																					
    pdfcenterwindow=true,																							
    colorlinks=true,        																		 			
    linkcolor=black,			
    citecolor=black,			
    filecolor=black,			
    urlcolor=black,				
}
\usepackage{scrhack}

\newtheorem{thm}{Theorem}[section]
\newtheorem{lem}[thm]{Lemma}
\newtheorem{prop}[thm]{Proposition}
\newtheorem{cor}[thm]{Corollary}

\newtheorem{main}{Main Theorem}

\theoremstyle{definition}
\newtheorem{df}[thm]{Definition}
\newtheorem{cons}[thm]{Construction}

\theoremstyle{remark}
\newtheorem{rem}[thm]{Remark}
\newtheorem*{remnn}{Remark}

\theoremstyle{definition}
\newtheorem{ex}[thm]{Example}

\newcommand{\IO}{\mathcal{O}}
\newcommand{\IN}{\mathbb{N}}
\newcommand{\IZ}{\mathbb{Z}}
\newcommand{\IQ}{\mathbb{Q}}

\newcommand{\IA}{\mathbb{A}}

\newcommand{\IP}{\mathbb{P}}

\newcommand{\A}{\mathcal{A}}

\newcommand{\Pc}{\mathcal{P}}
\newcommand{\PD}{\mathcal{P}d(\A)}
\newcommand{\G}{\mathcal{G}(\A)}
\newcommand{\GM}{\mathcal{G}m(\A)}

\newcommand{\bb}[1]{{\bm{#1}}} 

\newcommand{\HA}{\mathcal{H}_{\A}}
\newcommand{\Ip}{\mathpzc{p}}
\newcommand{\Iq}{\mathpzc{q}}

\newcommand{\iivec}[2]{\left(\begin{smallmatrix}#1\\ #2\end{smallmatrix}\right)}

\newcommand{\field}{\Bbbk}
\newcommand{\mon}{\mathcal{M}}
\newcommand{\UM}{\mathcal{U}_{\mon}} 
\newcommand{\JMP}{\widetilde{J_{\mon}(\Ip)}}
\newcommand{\EX}{\vspace*{-0.5cm}\begin{flushright}\ensuremath{\Diamond}\end{flushright}}

\newcommand{\rv}{\mathfrak{r}}

\newcommand{\vect}[2]{\iivec{#1}{#2}}

\usepackage[vcentering,dvips]{geometry}
\def\address#1#2{\begingroup
\noindent\parbox[t]{7.8cm}{%
\small{\scshape\ignorespaces#1}\par\vskip1ex
\noindent\small{\itshape E-mail address}%
\/: #2\par\vskip4ex}\hfill%
\endgroup}%

\author{Ren\'e Birkner }
\title{Non-coherent Components of the \\Toric Hilbert Scheme}
\date{}
\begin{document}
\maketitle


\begin{abstract}
	We want to understand the geometry of all irreducible components of the toric Hilbert scheme. Until now it is known that 
	the coherent component is (up to normalisation) the toric variety associated to the state polytope 
	of the toric ideal $I_{\A}$. For the non-coherent components it was only known that there exists 
	such a polytope describing the normalisation. Using the local equations and various facts about toric Hilbert schemes, 
	we will derive an explicit construction of the polytope corresponding to the normalisation of the underlying 
	reduced structure of a given non-coherent component of the toric Hilbert scheme.
\end{abstract}

\section{Introduction}                   %

	The toric Hilbert scheme has been studied by Arnol$'$d \cite{Arnold:Agraded}, Korkina, Post and Roelofs 
	\cite{Korkina:Classification,KorkinaPostRoelofs:Classification}, Sturmfels \cite{Sturmfels:GeomAGraded}, Peeva and Stillman 
	\cite{PeevaStillman:ToricHilbert,PeevaStillman:LocalEquations}, Maclagan and Thomas \cite{MaclaganThomas:RankTwoLattice,MaclaganThomas:CombTorHilbert}, 
	and others. It is given by the multigraded Hilbert function of the toric ideal $I_{\A}$ in $S\field[x_1,...,x_n]$. As defined in \cite{Sturmfels:Groebner} 
	an ideal isomorphic to such an initial ideal of the toric ideal is called \emph{coherent}, where isomorphic means that the corresponding algebras are isomorphic 
	as multigraded $\field$-algebras. Hence, the isomorphism classes of coherent $\A$-graded ideals are in bijection with the cones of the Gröbner fan of $I_{\A}$.
	
	The global equations of this scheme are in general very extensive, but Peeva and Stillman show that there is a cover of affine open charts around the 
	monomial $\A$-graded ideals and that the calculations of the local equations for these affine charts are much more feasible \cite{PeevaStillman:LocalEquations}. 
	They show that the toric Hilbert scheme can have several irreducible components. However, there is a unique component containing the toric 
	ideal \cite{PeevaStillman:ToricHilbert}, which in fact contains exactly all coherent $\A$-graded ideals. Therefore, this component is called the 
	\emph{coherent component} of the toric Hilbert scheme. Moreover, the normalisation of the coherent component is the toric variety associated to the Gröbner fan 
	of the toric ideal \cite{StillmanSturmfelsThomas:Algorithms}. 
	A construction by Sturmfels \cite{Sturmfels:GeomAGraded} shows that the underlying reduced structure of each irreducible component of the toric Hilbert 
	scheme is given by binomial equations so that all these reduced components are projective toric varieties and therefore each of them contains a dense torus, 
	the so-called \emph{ambient torus} of that component. Thus, this implies the existence of a polytope $P_V$ for each component $V$ of the toric Hilbert scheme 
	such that the normalisation of this component is the toric variety associated to the normal fan of $P_V$ (Corollary \ref{cor:EveryComponentHasPolytope}).
	
	This work presents an explicit construction of the polytope $P_V$ for an arbitrary non-coherent reduced component $V$ of a toric Hilbert scheme. For this we use 
	the local equations around a monomial $\A$-graded ideal $\mon$ in new variables $\bb{y}$ to construct a so-called \emph{universal family} $J_{\mon}(\Ip)$ for each 
	non-coherent irreducible component containing $\mon$ (Definition \ref{df:universalFamilyOfAComponent}), where $\Ip$ denotes an associated prime of the local 
	equations defining the underlying reduced structure $V_{\Ip}$ of such a non-coherent irreducible component. 
	
	\begin{main}[Theorem \ref{thm:universalfamilyofcomponent}]
	  	The universal family $J_\mon(\Ip)$ parametrises the ambient torus of the reduced irreducible non-coherent component $V_{\Ip}$ in the toric Hilbert scheme 
	  	over the points in $(\field^*)^{\dim(V_{\Ip})}$.
	  	To be precise, the closed points of this irreducible component of the toric Hilbert scheme intersected with its ambient torus are exactly those $\A$-graded 
	  	ideals that are given by substituting a point $(\lambda_i)_{i=1,...,\dim(V_{\Ip})} \in (\field^*)^{\dim(V_{\Ip})}$ into $J_\mon(\Ip)$.
	\end{main}
	
	Thus, the ambient torus $(\field^*)^{\dim(V_{\Ip})}$ of a non-coherent component is different from the torus $\Spec(\field[\bb{x}^{\pm 1}])$ of the coherent 
	component. Then we construct a homogenised version {\footnotesize $\JMP$} of $J_\mon(\Ip)$ by introducing a new set of variables 
	$z_1,...,z_{\dim(V_{\Ip})}$ with the same degrees as the $\bb{y}$ variables (Definition \ref{df:generalisedUniversalFamily}). This homogeneous 
	family {\footnotesize $\JMP$} is called the \emph{generalised universal family} and also gives the ambient torus of the non-coherent component. But more 
	importantly, it gives the main result:
	
	\begin{main}[Theorem \ref{thm:ThePolytope}]
	  	Let $\mon$ be a monomial $\A$-graded ideal and {\footnotesize $\JMP$} in $\field[\bb{x},y_i,z_i \,|\, i=1,...,\dim(V_{\Ip})]$ a generalised 
	  	universal family  of a reduced component $V_{\Ip}$ containing $\mon$. Then {\footnotesize $\JMP$} is homogeneous with respect to a strictly positive 
	  	grading and the normalisation of the component $V_{\Ip}$ is the toric variety defined by the normal fan of the state polytope 
	  	$\state(${\footnotesize $\JMP$}$)$, \textit{i.e.} the Gröbner fan of {\footnotesize $\JMP$}. 
	\end{main}
	
	A more detailed approach to this construction in combination with an overview on most of the results about toric Hilbert schemes can be found in the present author`s 
	doctoral thesis \cite{Birkner:Exotic}.

\section{Preliminaries}                  %

		We will be using the lattice $M:=\IZ^d$ with its associated vector space $M_{\IQ}:=M\otimes_{\IZ}\IQ$ over $\IQ$ and dual lattice 
		$N := \Hom_{\IZ}(M,\IZ)\cong \IZ^d$ with corresponding $\IQ$-vector-space $N_{\IQ}:=N\otimes_{\IZ}\IQ$. We will work over an 
		algebraically closed field $\field$. 
		
		Let $\A$ be a collection of $n$ vectors $a_1,...,a_n$ in $M$ such that $0\in M$ is not contained in their positive hull. One can also 
		define $\A$ as a linear map of lattices $\A : \IZ^n \rightarrow M$ by $e_i \mapsto a_i$ such that 
		\begin{equation}\label{eq:KerAcapN}
			\Ker(\A)\cap \IN^n = 0.
		\end{equation} 
		We will often make use of both notations. 
		We denote by $\IN\A$ the semigroup generated by $a_1,...,a_n$ in $M$. This is the image of $\IN^n$ under $\A$. Furthermore, we suppose that $\A$ has 
		rank $d$. Otherwise we can restrict $M$ to the sublattice $M\cap \A(\IQ^n)$ in which $\A$ has full rank. Define 
		a polynomial ring $S:=\field[x_1,...,x_n]$ over $\field$ with an $M$-grading given by $\A$. This just means that $x_i$ has degree $a_i$. 
		Furthermore, an element $r \in S$ is $M$-homogeneous if every term of $r$ has the same degree in the $M$-grading induced by $\A$, and an ideal 
		$I\subseteq S$ is $M$-homogeneous if for every $r\in I$ and its decomposition $r=\sum r_a$ into $M$-degree parts, such that $r_a\in I$ holds for each $a$, 
		where $r_a$ is the sum over all terms in $r$ with degree $a\in M$.
		
		\begin{df}
			An ideal $I\subseteq S$ is called \emph{$\A$-graded} if it is $M$-homogeneous and 
			\begin{equation}
				\Dim_\field\left( S/I \right)_a = \left\{ \begin{array}[]{ll}
																								1 & \textrm{if } a \in \IN\A \\
																								0 & \textrm{otherwise} 
																							\end{array}\right.
			\end{equation}
			holds for its multigraded Hilbert function. We call a $\field$-algebra $\A$-graded if it is of the form $S/I$ for some $\A$-graded ideal $I$.
		\end{df}
			
		A special $\A$-graded ideal is the \emph{toric ideal} $I_{\A}$, which is the kernel of the $\field$-algebra homomorphism $S \rightarrow \field[t_1^{\pm 1},...,t_d^{\pm 1}]$ 
		that maps $x_i$ to $\bb{t}^{a_i}=t_1^{a_i^1}\cdot \ldots \cdot t_d^{a_i^d}$ for $1\leq i \leq n$. Therefore, 
		$I_{\A}=\left\langle \bb{x}^{\bb{a}}-\bb{x}^{\bb{b}} \,|\, \bb{a},\bb{b} \in \IN^n, \bb{a}-\bb{b} \in \Ker(\A)\right\rangle$ is generated 
		by binomials and is by definition $\A$-graded, because it identifies all monomials of the same degree. The toric ideal gives rise to a particular class of 
		$\A$-graded ideals.
		
		\begin{df}
			An $\A$-graded ideal $I$ is called \emph{coherent} if there is a degree $0$ isomorphism between $S/I$ and $S/I_{\A}$, \textit{i.e.} there exists 
			$\lambda \in \left(\field^*\right)^n$ such that $\lambda.I = I_{\A}$.
		\end{df}
		
		\begin{df}\label{definition:Graver}
			A binomial $\bb{x}^{\bb{u}}-\bb{x}^{\bb{v}}\in I_{\A}$ is called \emph{primitive} (or \emph{Graver}) if there are no proper monomial factors $\bb{x}^{\bb{u}'}$ 
			of $\bb{x}^{\bb{u}}$ and $\bb{x}^{\bb{v}'}$ of $\bb{x}^{\bb{v}}$ with $\bb{x}^{\bb{u}'}-\bb{x}^{\bb{v}'}\in I_{\A}$. A degree $\bb{a} \in \IN\A$ is called a 
			\emph{primitive} (or \emph{Graver}) \emph{degree} if there exists some primitive binomial $\bb{x}^{\bb{u}}-\bb{x}^{\bb{v}}$ with degree $\A \bb{u}=\A \bb{v}= \bb{a}$ and 
			we denote the set of all primitive degrees by $\PD$. The set of all primitive binomials is called the \emph{Graver basis} and we denote it by $\G$.
		\end{df}
		
		In \cite{MaclaganThomas:CombTorHilbert} and \cite{PeevaStillman:ToricHilbert}, a scheme is constructed that parametrises 
		all $\A$-graded ideals, the so-called toric Hilbert scheme. We will give the definition of Peeva and Stillman \cite{PeevaStillman:ToricHilbert} 
		in which they use Fitting ideals (see \cite[Section 20.2]{Eisenbud:CommAlg}). Let $\bb{a} \in \PD \subset \IN\A$ be a primitive degree. We denote 
		the number of elements in the fiber of $\bb{a}$ under the map $\A$ by $|\bb{a}|+1$ and the set of all such monomials by $\GM_{\bb{a}}$, and assume 
		that we have ordered them. The Graver basis of $\A$ is finite, so that we have $\PD=\left\{\bb{a}_1,...,\bb{a}_l\right\}$. From this we define the 
		following product of projective spaces over $\field$
		\[\Pc:=\IP_\field^{|\bb{a}_1|} \times ... \times \IP_\field^{|\bb{a}_l|}.\]
		Now consider the subset $\mathcal{Y}\subset \Pc \times \IA_\field^n$ given by
		\[I(\mathcal{Y})=\left\langle \left.\xi^{\bb{a}}_{\bb{n}} \cdot \bb{x}^{\bb{m}} - \xi^{\bb{a}}_{\bb{m}} \cdot \bb{x}^{\bb{n}} \,\right|\, 
																\forall \bb{a} \in \PD \textrm{ and } \bb{x}^{\bb{m}},\bb{x}^{\bb{n}}\in \GM_{\bb{a}}\right\rangle\]
		with projection $\phi : \mathcal{Y} \rightarrow \Pc$ and the grading on $\IO_{\mathcal{Y}}$ induced by the $M$-grading given by $\A$ on $S$. Hence, $\phi^\#$ 
		is $M$-homogeneous and therefore we can write 
		\[\mathcal{Y}=\Spec_{\Pc}\left(\bigoplus_{\bb{a}\in \IN\A} L_{\bb{a}}\right)\]
		where $L_{\bb{a}}$ are coherent $\IO_{\Pc}$-modules and $L_0=\IO_{\Pc}$ (see \cite[Definition 3.1]{PeevaStillman:ToricHilbert}). 
		We define an ideal of $\IO_{\Pc}$
		\[\dets(\phi) = \sum_{\bb{a} \in \IN\A} \Fitt_0(L_{\bb{a}}),\]
		where $\Fitt_0(L_{\bb{a}})$ is the $0$-th Fitting ideal of $L_{\bb{a}}$. For details on Fitting ideals we refer the reader to \cite{Eisenbud:CommAlg}
		
		\begin{df}
			The \emph{toric Hilbert scheme} is defined as
			\[\HA = V\left(\dets(\phi)\right) \subset \Pc.\]	
		\end{df}
		
		\begin{rem}\label{rem:THSCorrespondence}
			The parametrisation of the $\A$-graded ideals is as follows: Each closed point of $V\left(\dets(\phi)\right)$ is a product of points
			\[\left(\lambda^{\bb{a}_1},...,\lambda^{\bb{a}_l}\right) \in \IP_\field^{|\bb{a}_1|} \times ... \times \IP_\field^{|\bb{a}_l|}\]
			such that
			\[\left\langle\lambda^{\bb{a}}_{\bb{n}}\bb{x}^{\bb{m}}-\lambda^{\bb{a}}_{\bb{m}}\bb{x}^{\bb{n}}\,|\,
						\forall \bb{a} \in \PD \textrm{ and } \bb{x}^{\bb{m}},\bb{x}^{\bb{n}}\in \GM_{\bb{a}}\right\rangle\]
			is the corresponding $\A$-graded ideal.
		\end{rem}
		
		In \cite[Chapter 5]{Sturmfels:GeomAGraded} Sturmfels has constructed a parameter space of $\A$-graded ideals which is not given by determinantal equations. 
		He considers the projective product space $\Pc':=\prod \IP^{|\bb{a}|}$ for all $\bb{a}\in Z_r(\A)\cap \IN\A$ where 
		\[Z_r(\A):= \left\{\left. \sum_{i=1}^n \lambda_i\cdot a_i \,\right|\, 0 \leq \lambda_i \leq r \textrm{ for } i=1,...,n, \lambda_i \in \IQ\right\}\subset M_{\IQ}\]
		is the zonotope with edge length $r=(n-d)^{2^n}\cdot a^{d2^n}$. With the same notation for $\xi \in \Pc'$ as before in $\Pc$ he defined the closed 
		subscheme $\Pc_{\A} \subset \Pc'$ by the equations 
		\[ \xi^{\bb{a}}_{\bb{m}_1} \cdot \xi^{\bb{a}+\bb{b}}_{\bb{m}_2+\bb{n}} = \xi^{\bb{a}}_{\bb{m}_2} \cdot \xi^{\bb{a}+\bb{b}}_{\bb{m}_1+\bb{n}},\]
		whenever $\Deg(\bb{m}_1)=\Deg(\bb{m}_2)=\bb{a}$ and $\Deg(\bb{n})=\bb{b}$. This is also a description of all $\A$-graded ideals.
		
		\begin{thm}
			There exists a natural bijection between the set of $\A$-graded ideals in $S$ and the set of closed points of $\Pc_{\A}$.
		\end{thm}
		
		\begin{proof}
			See \cite[Theorem 5.3]{Sturmfels:GeomAGraded}.
		\end{proof}
		
		In \cite{HaimanSturmfels:Multigraded} Haiman and Sturmfels give general constructions of different multigraded Hilbert schemes. 
		In particular, their work shows that the toric Hilbert scheme $\HA$ by Peeva and Stillman and the parameter space $\Pc_{\A}$ by Sturmfels 
		are in fact the same, \textit{i.e.} $\HA\cong \Pc_{\A}$ holds. For this use \cite[Propositions 5.2 + 5.3]{HaimanSturmfels:Multigraded} in combination with 
		\cite[Theorem 3.16]{HaimanSturmfels:Multigraded}. Thus, we get the following:
		
		\begin{lem}\label{lem:THSbinomialequations}
			The toric Hilbert scheme $\HA$ is given by binomial equations.\qed
		\end{lem}
		
		If we use the primary decomposition theorem from the work of Eisenbud and Sturmfels on binomial ideals 
		\cite[Theorem 7.1]{EisenbudSturmfels:BinomialIdeals} it follows that every irreducible component of $\HA$ is generated by binomial 
		ideals. Since the radical of a binomial ideal is again a binomial ideal (see \cite[Theorem 3.1]{EisenbudSturmfels:BinomialIdeals}) 
		the reduced structure of each irreducible component, \textit{i.e.} the variety given by the radicals of a covering of local rings, is given 
		by binomial equations. This argument proves the following lemma:
		
		\begin{lem}\label{lem:EveryComponentisToric}
			The underlying reduced structure of each irreducible component of the toric Hilbert scheme is a (not necessarily normal) projective 
			toric variety.\qed
		\end{lem}
		
		\begin{cor}\label{cor:EveryComponentHasPolytope}
			For each irreducible component $V$ of the toric Hilbert scheme there is a polytope $P_V$ such that the projective variety of the normal 
			fan of $P_V$ is the normalisation of $V$.\qed
		\end{cor}
		
		\begin{df}\label{df:ambientTorus}
			We call the dense torus of an irreducible component $V$ of $\left(\HA\right)_{\textrm{red}}$ the \emph{ambient torus} of $V$.
		\end{df}
		
		\begin{lem}\label{lem:diagonaltorusaction}
			Any torus acting on the underlying reduced structure of an irreducible component of a Toric Hilbert Scheme $\HA$ acts diagonally 
			by scaling each coordinate.
		\end{lem}
		
		\begin{proof}
			The global equations of $\HA$ are a binomial ideal by Lemma \ref{lem:THSbinomialequations}. It follows from 
			\cite[Theorem 7.1]{EisenbudSturmfels:BinomialIdeals} that the associated primary ideals are also binomial ideals. Thus, the irreducile 
			components are given by binomial equations. Using \cite[Theorem 3.1]{EisenbudSturmfels:BinomialIdeals} we get that also their radicals 
			are binomials ideals. Hence, the underlying reduced irreducible structure of an irreducible component is given by binomial equations. 
			Therefore, any torus must act diagonally.
		\end{proof}
		
		\begin{thm}\label{thm:theCoherentComponent}
			The toric ideal $I_{\A}$ lies on a unique irreducible component of the toric Hilbert scheme $\HA$, the 
			\emph{coherent component}. The normalisation of the coherent component is the projective toric variety defined by the Gröbner 
			fan of $I_{\A}$.
		\end{thm}
		
		\begin{proof}
			See \cite[Theorem 4.1]{StillmanSturmfelsThomas:Algorithms}.
		\end{proof}
		
		Thus, the coherent component is well known and described. For the other ones, the \emph{non-coherent components}, we need local equations around 
		monomial $\A$-graded ideals for the toric Hilbert scheme. These have been presented in \cite{PeevaStillman:ToricHilbert}.	Fix some monomial 
		$\A$-graded ideal $\mon$. Then for every degree $\bb{a} \in \IN\A$ there is a unique monomial $s_{\bb{a}}$ which is not in $\mon$. This is called 
		the \emph{$\mon$-standard monomial} of degree $\bb{a}$.
		
		\begin{df}
			Let $\UM \subset \HA$ be the affine open subscheme
			\[\UM := \HA \cap \left\{ \left.\xi_{s_{\bb{a}}}^{\bb{a}}\neq 0 \,\right|\, \bb{a} \in \PD\right\}.\]
		\end{df}
		
		This means we have chosen an affine chart for every $\IP^{|\bb{a}|}$ in $\Pc$ and intersected these with the toric Hilbert 
		scheme. Since it follows from the construction of the toric Hilbert scheme that $\bb{x}^{\bb{m}} \in \mon$ if $\xi_{\bb{m}}^{\bb{a}}=0$ 
		(see Remark \ref{rem:THSCorrespondence}), it follows that $\mon$ is contained in $\UM$.
		
		\begin{rem}\label{rem:openCoverStandard}
			The affine open subscheme $\UM$ corresponds exactly to those $\A$-graded ideals with the same standard monomials as $S/\mon$.
		\end{rem}		
		
		\begin{lem}
			The set $\left\{ \UM \right\}$ is an affine open cover for $\HA$, where $\mon$ runs over all monomial $\A$-graded ideals.\qed
		\end{lem}
		
		\begin{proof}
			Consider an arbitrary $\A$-graded ideal $I$. Take any initial monomial ideal $\mon$ of $I$. Then the $\mon$-standard monomial 
			$s_{\bb{a}}$ of degree $\bb{a}$ of $\mon$ forms a vector space basis of $\left( S/I \right)_{\bb{a}}$ for every 
			$\bb{a} \in \IN\A$. Since therefore $s_{\bb{a}}$ is not contained in $I$, which means $\xi_{s_{\bb{a}}}^{\bb{a}}\neq 0$ for 
			all $\bb{a} \in \PD$, we have $I \in \UM$.
		\end{proof}
		
		\begin{lem}\label{lem:localChartcontainsTorus}
			Let $V\subseteq \HA$ be an irreducible component containing $\mon$. Then $\left(\UM\cap V\right)_{\textrm{red}}$ contains the ambient torus of 
			$V_{\textrm{red}}$.
		\end{lem}
		
		\begin{proof}
			We may assume that $V$ is already reduced. Let $\xi$ be a point on the ambient torus of $V$. Then the corresponding $\A$-graded ideal is
			\[I_{\xi} =\left\langle \xi_{\bb{n}}^{\bb{a}}\bb{x}^{\bb{m}}-\xi_{\bb{m}}^{\bb{a}}\bb{x}^{\bb{n}}\,|\,\bb{x}^{\bb{m}},\bb{x}^{\bb{n}}\in \GM_{\bb{a}}, \bb{a} \in \PD \right\rangle.\]
			By Lemma \ref{lem:diagonaltorusaction} the ambient torus acts diagonally on the coefficients $\xi_{\bb{n}}^{\bb{a}},\xi_{\bb{m}}^{\bb{a}}$. Hence, 
			if $\bb{x}^{\bb{m}_0}\in I_{\xi}$ then $\bb{x}^{\bb{m}_0}\in I$ holds for every $I$ in the ambient torus orbit of $I_{\xi}$. 
			Moreover, $\bb{x}^{\bb{m}_0}\in I_{\xi}$ implies $\bb{x}^{\bb{m}_0}\in I'$ for every $I'$ in the closure of the ambient torus orbit.
			
			On the other hand, since $\mon$ lies on $V$, \textit{i.e.} the closure of the ambient torus, $\bb{x}^{\bb{m}_0}\notin \mon$ implies $\bb{x}^{\bb{m}_0}\notin I$ 
			for every $I$ in the ambient torus. Thus, the monomials of $S/\mon$ are standard monomials for all ideals in the ambient torus of $V$. Hence 
			by Remark \ref{rem:openCoverStandard} the ambient torus lies in $\UM$.
		\end{proof}
		
		\begin{lem}\label{lemma:UMring}
			Let $\mon$ be a monomial $\A$-graded ideal. For the affine open chart $\UM$	and the corresponding set of variables 
			$Z=\left\{\left. \xi_{\bb{m}}^{\bb{a}} \,\right|\, \bb{a} \in \PD, \deg(\bb{x}^{\bb{m}})=\bb{a}, \bb{x}^{\bb{m}}\neq s_{\bb{a}} \right\}$ set
			\[\renewcommand{\arraystretch}{1.35}\begin{array}{rl}
				G := & \left\langle \left.\bb{x}^{\bb{m}} - \xi_{\bb{m}}^{\bb{a}}\cdot s_{\bb{a}} 
												\,\right|\, \bb{a} \in \PD, \deg(\bb{x}^{\bb{m}})=\bb{a}, \bb{x}^{\bb{m}}\neq s_{\bb{a}}\right\rangle\subseteq \field[Z]\otimes_{\field}S \quad \textrm{and}\\
				F := & \sum_{\bb{a} \in \IN\A} \Fitt_0\left(\left( \field[Z][x_1,...,x_n]/G\right)_{\bb{a}}\right)\subseteq \field[Z].
			\end{array}\renewcommand{\arraystretch}{1}\]
			Then
			\[\UM = \Spec\left(\field[Z]/F\right).\]
			In particular, the ideal $F$ is generated by the maximal minors of matrices of the form
			\[\left(\begin{array}{cc} 
								\begin{array}{cc} 
									\begin{array}{cc} 
										1 & \\ & 1
									\end{array} & 
									0\\
									0 & 
									\begin{array}{cc} 
										\ddots & \\ & 1 
									\end{array} 
								\end{array} & 
								0 \\ 
								\begin{array}{cccc}								
									r^{\bb{a}}_0 & r^{\bb{a}}_1 & \cdots & r^{\bb{a}}_{|\bb{a}|-1}
								\end{array} &
								\begin{array}{ccc}
									r^{\bb{a}}_{|\bb{a}|} & \cdots & r^{\bb{a}}_t
								\end{array}
							\end{array}\right),\]			
			(where we assume that $s_{\bb{a}}=\bb{x}^{\bb{m}_{|\bb{a}|}}$). Therefore, we have 
			\[F= \sum_{\bb{a} \in \IN\A} \left( G : s_{\bb{a}} \right)\]
			and thus $F$ is a binomial ideal.
		\end{lem}
		
		\begin{proof}
			See \cite[Corollary 4.5]{PeevaStillman:ToricHilbert}.
		\end{proof}
		
		Now suppose that $\mon$ is a coherent monomial $\A$-graded ideal. We will give an efficient description of $\field[Z]/F$ 
		constructed by Peeva and Stillman.
		
		\begin{cons}[\textbf{Local coherent equations}]\label{cons:localcoherentequations}
			The ideal $\mon$ has a unique minimal set of monomial generators. We call this set 
			$G_\mon=\left\{f_i \,|\, 1\leq i \leq p_\mon\right\}$ where $p_\mon$ is the number of generators of $\mon$. Then for every $f_i$ there 
			is an $\mon$-standard monomial of degree $\Deg(f_i)$ which we call $s_i$. Note that $f_i-s_i$ is primitive because otherwise $f_i$ would 
			not be a minimal generator of $\mon$. Consider the ring $\field[x_1,...,x_n,y_1,...,y_{p_\mon}]$ and the ideal 
			$J_\mon$ generated by the set
			\[\overline{G_\mon} = \left\{\left. f_i-y_i\cdot s_i \,\right|\, 1\leq i \leq p_\mon \right\}. \]
			We fix a term order $\prec_{\bb{x}}$ on $S$ such that $\mon=\init_{\prec_{\bb{x}}}(I_{\A})$ and an arbitrary term order 
			$\prec_{\bb{y}}$ on $\field[y_1,...,y_{p_\mon}]$. Denote by $\prec$ the product term order of $\prec_{\bb{x}}$ and of
			$\prec_{\bb{y}}$ on $\field[x_1,...,x_n,y_1,...,y_{p_\mon}]$, which means 
			\[\bb{x}^{\bb{a}} \cdot \bb{y}^{\bb{b}} \succ \bb{x}^{\bb{a}'} \cdot \bb{y}^{\bb{b}'} \Leftrightarrow 
							\bb{x}^{\bb{a}} \succ_{\bb{x}} \bb{x}^{\bb{a}'} \textrm{ or } \bb{x}^{\bb{a}} = \bb{x}^{\bb{a}'}, \bb{y}^{\bb{b}} \succ_{\bb{y}} \bb{y}^{\bb{b}'}.\]
			For each pair of binomials $u$ and $v$ in $\overline{G_\mon}$ we form their S-polynomial 
			\[s(u,v) := \frac{h}{\init_{\succ}(u)}u - \frac{h}{\init_{\succ}(v)}v,\]
			where $h$ is the least common multiple of $\init_{\succ}(u)$ and $\init_{\succ{v}}$. This is a homogeneous 
			binomial with respect to the $M$-grading induced by $\A$ on $\field[x_1,...,x_n]$. Then we choose a reduction of $s(u,v)$ by $\overline{G_\mon}$ 
			to $\left( e(\bb{y})-h(\bb{y})\right)\cdot s_{u,v}$, where $e$ and $h$ are monomials in $\field[y_1,...,y_{p_\mon}]$. Since 
			$\prec_{\bb{x}}$ is a term order with $\mon=\init_{\prec_{\bb{x}}}(I_{\A})$, $s_{u,v}$ is the $\mon$-standard monomial of degree of 
			$s(u,v)$. Set $r(u,v):=e(\bb{y})-h(\bb{y}) \in \field[y_1,...,y_{p_\mon}]$ and define 
			\[I_\mon:=\left\langle r(u,v) \,\left|\, u,v \in \overline{G_\mon}\right.\right\rangle\subset \field[y_1,...,y_{p_\mon}].\]
		\end{cons}
		
		\begin{thm}\label{thm:localCoherentEquations}
			Let $\mon$ be a coherent monomial $\A$-graded ideal. Then 
			\[\UM = \Spec\left( \field[Z]/F\right) \cong \Spec\left( \field[y_1,...,y_{p_\mon}]/I_\mon \right).\]
		\end{thm}
		
		\begin{proof}
			See \cite[Theorem 3.2]{PeevaStillman:LocalEquations}.
		\end{proof}
		
		In the non-coherent case the description of $\UM$ is not so easy because we do not have a term order for which $\mon$ is the initial 
		ideal of $I_{\A}$. We can still compute $\field[Z]/F$ using Lemma \ref{lemma:UMring}, but Construction \ref{cons:localcoherentequations}, 
		which would be much more efficient cannot be used, since the proof uses the term order $\prec_{\bb{x}}$ and the reduction with respect 
		to it. But Peeva and Stillman even enhanced this construction to the non-coherent case. They calculated the local ring 
		of $\HA$ in the point $\mon$.
		
		\begin{thm}\label{thm:localringatnoncoherentmonomial}
			The local ring of $\HA$ at $\mon$ is
			\[\IO_{\HA,[\mon]} \cong \field[Z]_{\left\langle Z\right\rangle}/F\cong \field[y_1,...,y_{p_\mon}]_{\left\langle y_1,...,y_{p_\mon}\right\rangle}/I_\mon.\]
		\end{thm}
		
		\begin{proof}
			See \cite[Theorem 4.4]{PeevaStillman:LocalEquations}.
		\end{proof}
		
		There is even a similar construction to Construction \ref{cons:localcoherentequations} that uses Mora's tangent cone algorithm (see \cite{Ferdinando:Algorithm}) in a simplified way instead 
		of the Gröbner reduction by $\overline{G_\mon}$. We end this chapter by giving the construction of $I_\mon$ for a non-coherent monomial $\A$-graded ideal $\mon$.
		
		\begin{cons}[\textbf{Local non-coherent equations}]\label{cons:localincoherentequations}
			The first steps are exactly as in Construction \ref{cons:localcoherentequations}. There is again a minimal set of monomials  
			$G_\mon:=\left\{f_i \,|\, 1\leq i \leq p_\mon\right\}$ generating $\mon$, where $p_\mon$ is the number of generators of $\mon$. Let $s_i$ be the $\mon$-standard monomial of degree 
			$\Deg(f_i)$. Still $f_i-s_i$ is primitive by the same argument as before. Consider the ring $\field[x_1,...,x_n,y_1,...,y_{p_\mon}]$ and the ideal 
			$J_\mon$ generated by the set
			\[\overline{G_\mon} = \left\{\left. f_i-y_i\cdot s_i \,\right|\, 1\leq i \leq p_\mon \right\}. \]
			Now we can not fix a term order as before. We have to use the second reduction process from Peeva and Stillman in \cite{PeevaStillman:LocalEquations}.
			
			Fix an order $\prec$ on the monomials of $\field[x_1,...,x_n,y_1,...,y_{p_\mon}]$ with $y_i\prec 1 \prec x_j$ for all $i,j$. 
			Note that this is not a term order since $1$ is not the minimal element. Then $f_i$ is the initial term of each element in 
			$\overline{G_\mon}$. Let $m$ be a monomial in $\field[x_1,...,x_n,y_i,...,y_{p_\mon}]$. Then the \emph{remainder} $R(m,\overline{G_\mon})$ 
			is constructed as follows. If $m$ is not divisible by any of the monomials $f_i$ then $R(m,\overline{G_\mon})=m$, otherwise we have $m=f_i\cdot u$ for some $i$ and 
			monomial $u$. Then we reduce $m$ to $m_1:=u\cdot y_i\cdot s_i$. We repeat this reduction until at some point we either get at some point some $m_p$ that is not further 
			reducible by that method, in which case we set $R(m,\overline{G_\mon})=m_p$, or we obtain a loop
			\[m \rightarrow m_1\rightarrow m_2 \rightarrow ... \rightarrow m_i \rightarrow ... \rightarrow m_j \rightarrow ...\]
			where $m_i$ divides $m_j$. Then we set $R(m,\overline{G_\mon})=0$. This reduction is extended to polynomials by linearity. Note that 
			the remainder of any monomial is either $0$ or $\bb{y}^{\bb{e}}\cdot s_a$ for some standard monomial $s_a$ and $\bb{e}\in\IN^{p_{\mon}}$.
			
			For each pair of binomials $u$ and $v$ in $\overline{G_\mon}$ we form their S-polynomial $s(u,v)$ as in Construction \ref{cons:localcoherentequations} and set 
			\[r(u,v):=R(s(u,v),\overline{G_\mon})/s_{u,v},\]
			where $s_{u,v}$ is the standard monomial of degree of $s(u,v)$. Then 
			\[I_\mon:=\left\langle r(u,v) \,\left|\, u,v \in \overline{G_\mon}\right.\right\rangle\subset \field[y_1,...,y_{p_\mon}].\]
		\end{cons}

\section{Universal Families}           %

		We will now construct a so-called universal family that parametrises the ambient torus of a non-coherent component of the toric Hilbert scheme. 
		For this let $\mon$ be a monomial $\A$-graded ideal. At first, assume that $\mon$ is coherent. Then by Theorem \ref{thm:localCoherentEquations} 
		we can compute the local equations $I_\mon \subseteq \field[y_1,...,y_l]$ of $\HA$ around $\mon$ where $l$ is the number of generators 
		of $\mon=\left\langle \bb{x}^{\bb{m}_1},...,\bb{x}^{\bb{m}_l}\right\rangle$.
		
		\begin{df}
			We call the ideal
			\[J_\mon = \left\langle \bb{x}^{\bb{m}_1} - y_1\cdot s_1,...,\bb{x}^{\bb{m}_l} - y_l\cdot s_l \right\rangle\]
			from Construction \ref{cons:localcoherentequations} the \emph{universal family} of $\UM$ with \emph{defining ideal} $I_\mon$.
		\end{df}
		
		Note that by Lemma \ref{lemma:UMring} the $\A$-graded ideals that correspond to the points in 
		$\UM$ are precisely given by $J_\mon$ for all $(y_1,...,y_l)$ in the variety of $I_\mon$. 
		
		We will now outline the construction of a new universal family that describes the ambient torus of the underlying reduced structure of a 
		non-coherent component containing $\mon$. This is done in several steps. Firstly, we remove redundant variables from $I_{\mon}$ 
		and $J_{\mon}$ (Construction \ref{cons:reduceRedundantVariables}). Then we construct the primary decomposition of the 
		resulting defining ideal $I_{\mon}'$ to get the primary ideals $\Iq$ defining the irreducible components containing $\mon$ 
		(Propositions \ref{prop:coherentprimaryideals} and \ref{prop:incoherentprimaryideals}). Because the underlying reduced structure 
		of each component is a projective toric variety we take the radical $\Ip=\sqrt{\Iq}$ for each of these primary 
		ideals (Definition \ref{df:ComponentofAnAssociatedPrime}). Then we set the variables that are generators of $\Ip$ 
		to zero in the universal family $J_{\mon}'$ (Construction \ref{cons:removeSingleVariables}). Now the prime ideal 
		$\Ip$ has become a pure binomial ideal, \textit{i.e.} containing no monomials, and we perform a change of the $\bb{y}$ 
		coordinates in $J_{\mon}'$ so that $\Ip$ becomes trivial. By this we get a universal family that describes the ambient torus of
		the reduced structure of that non-coherent component without any defining ideal (Construction \ref{cons:Isomorphism}). 
		Let us go through this construction step by step.
		
		As seen in Construction \ref{cons:localcoherentequations}, the ideal $I_\mon$ is a binomial ideal, which 
		also follows from Lemma \ref{lem:THSbinomialequations}.  Hence, the generators of $I_\mon$ may contain binomials 
		of the form
		\[y_i - \prod_{j\neq i} y_j^{b_j}\]
		for some exponents $b_j$. Then we call the single variable $y_i$ a \emph{redundant variable}, since we can remove 
		$y_i$ from $I_\mon$ and $J_\mon$ by substituting $y_i$ by the product $\prod_{j\neq i} y_j^{b_j}$.
		
		\begin{cons}[\textbf{Removing redundant variables}]\label{cons:reduceRedundantVariables}
			Let $J_\mon$ be the universal family of a neighbourhood $\UM$ with defining ideal $I_\mon$. Let $y_i$ be a redundant 
			variable given by
			\[y_i - \prod_{j\neq i} y_j^{b_j} \in I_\mon.\]
			Then we \emph{remove the redundant variable} from $I_\mon$ and $J_\mon$ with the maps
			\[\Phi_i : \field[\bb{y}] \rightarrow \field[y_j \,|\, j \neq i],\, y_j \mapsto \left\{\begin{array}{cc} 
																					y_j & \textrm{if } j \neq i\\ 
																					\prod_{j\neq i} y_j^{b_j} & \textrm{if } j = i
																						\end{array}\right.\]
			and $\Psi_i = \Id_{\field[\bb{x}]} \otimes \Phi_i : \field[\bb{x},\bb{y}] \rightarrow \field[\bb{x}, y_j \,|\, j \neq i]$, respectively.
			
			We repeat this until there are no more redundant variables left. Then we denote by $\rv\subseteq \{1,...,l\}$ the indices 
			of the \emph{remaining variables} in $I_\mon$ and $J_\mon$ and write $I_\mon'$ and $J_\mon'$ for the ideals obtained by removing 
			the redundant variables in $I_\mon$ and $J_\mon$ respectively. This means we have
			\begin{eqnarray*}
				I_\mon' & \subseteq & \field[y_i \,|\, i \in \rv]\quad \textrm{and}\\
				J_\mon' & = & \left\langle \left.\bb{x}^{\bb{m}_j}- p_j(\bb{y})\cdot s_j\,\right|\, j=1,...,l\right\rangle,
			\end{eqnarray*}
			where $p_j(\bb{y})$ is the monomial into which $y_j$ has been converted by removing all redundant variables.
		\end{cons}		
		
		\begin{rem}\label{rem:IMintoJM}
			The points in $\UM$ are still completely described by substituting a solution of $I_\mon'$ for the $\bb{y}$ variables in $J_\mon'$.
		\end{rem}
		
		Still $I_\mon'$ gives only a parametrisation of the neighbourhood $\UM$ of $\mon$ and not of the different components that 
		contain $\mon$. However, if we decompose $I_\mon'$ into its primary ideals, then each of these primary ideals determines exactly one 
		of the possibly embedded components of $\HA$ intersecting $\UM$. 
		
		\begin{lem}
			Let $\mon$ be coherent and $I_{\mon}'=\bigcap \Iq_i$ be a minimal primary decomposition, then every component $V(\Iq_i)\subset \UM$ 
			contains $\mon$.
		\end{lem}
		
		\begin{proof}
			Fix a primary ideal $\Iq_i$ and take a point $\mu\in V(\Iq_i)$. This gives an $\A$-graded ideal
			\[I=\left(J_{\mon}\right)_{\left(\bb{y}=\mu\right)}=\left\langle \left. \bb{x}^{\bb{m}_j} - p_j(\mu)s_j\,\right|\, j = 1,...,l\right\rangle\]
			lying on the component $V$ given by $\Iq_i$. Recall the action of the $n$-torus $T=\left(\field^*\right)^n$ on $S=\field[x_1,...,x_n]$ by
			\[\lambda.x_i=\lambda_i x_i\]
			for $\lambda \in T$ which maps $\A$-graded ideals to $\A$-graded ideals. Hence, $T$ acts on $\HA$ and the orbit of a point under the 
			$T$-action lies in the same irreducible component as the point. Thus, the $T$-orbit of $I$ lies in $V$. Furthermore, $\mon$ was coherent so that 
			there exists some $\omega \in \IN^n$ such that $\mon=\init_{\omega}\left(I_{\A}\right)$. Finally, because 
			$\left\{\left. \bb{x}^{\bb{m}_j}-s_j\,\right|\,j=1,...,l\right\}$ is the reduced Gröbner basis with respect to $\omega$ we get that
			\[\mon = \left\langle \left.\bb{x}^{\bb{m}_j} \,\right|\, j= 1,...,l\right\rangle \subseteq \init_{\omega}\left(I\right) \]
			which is in fact an equality because both ideals are $\A$-graded. This implies that $\mon$ lies in the closure of the $T$-orbit of $I$ by 
			\cite[Theorem 15.17]{Eisenbud:CommAlg} and thus it lies on $V$.
		\end{proof}
		
		So let the primary decomposition be
		\[ I_\mon' = \Iq_1 \cap ... \cap \Iq_k.\]
	  Then $V(\Iq_i)$ is isomorphic to an irreducible affine subset of $\HA$, in fact of $\UM$, containing $\mon$. In particular, the closed points of 
	  $V(\Iq_i)$ substituted for the $\bb{y}$ variables in $J_\mon'$ give exactly all $\A$-graded ideals in that component intersected with $\UM$ which are the 
	  closed points of that component. This gives us the following proposition:
	  
	  \begin{prop}\label{prop:coherentprimaryideals}
	  	Let $\mon$ be a coherent monomial $\A$-graded ideal with local equations $I_\mon'$ and universal family $J_\mon'$, both after removing 
	  	redundant variables. Let $I_\mon' = \Iq_1 \cap ... \cap \Iq_k$ be a primary decomposition. Then $\overline{V(\Iq_i)}\subseteq \HA$ 
	  	is an irreducible component containing $\mon$ for $i=1,...,k$ and one of them is the coherent component. Furthermore, 
	  	$\overline{V(\Iq_i)}$ is the coherent component if and only if $\Iq_i$ contains no monomials.
	  \end{prop}
	  
	  \begin{proof}
	  	The first statement follows from Remark \ref{rem:IMintoJM} and that $V(\Iq_i)$ is an affine open subset of an irreducible component containing $\mon$, 
	  	which by Lemma \ref{lem:localChartcontainsTorus} is dense, so that $\overline{V(\Iq_i)}$ is in fact the component. Furthermore, $\mon$ is coherent so 
	  	one component must be the coherent one. For the last part, since $\Iq_i$ is generated by binomial differences, it contains no monomials 
	  	exactly if $(1,...,1)$ is in $V(\Iq_i)$. But this point corresponds to the $\A$-graded ideal
	  	\[\left\langle \bb{x}^{\bb{m}_1}-s_1,...,\bb{x}^{\bb{m}_l}-s_l\right\rangle,\]
	  	which is the toric ideal $I_{\A}$, because $\left\{\bb{x}^{\bb{m}_1}-s_1,...,\bb{x}^{\bb{m}_l}-s_l\right\}$ is a Gröbner basis of $I_{\A}$ with respect to 
	  	a term order giving $\mon$ as initial ideal. Thus $V(\Iq_i)$ contains the orbit of $I_{\A}$ under the action of the torus $T=(\field^*)^n$ which is the 
	  	torus of the coherent component. Since $I_{\A}$ only lies on the coherent component, the closure of $V(\Iq_i)$ is the coherent component.
	  \end{proof}
	  
	  Since the coherent component is already completely described by the state polytope of the toric ideal, we can ignore the primary ideal that corresponds to the 
	  coherent component and just consider the remaining $\Iq_j$'s containing at least one monomial generator.
	  
	  Now assume that $\mon$ is a non-coherent $\A$-graded ideal and we have computed $J_\mon$ and the local equations $I_\mon$ using Construction 
	  \ref{cons:localincoherentequations}.
	  
	  \begin{df}
	  	Let $\mon$ be a non-coherent monomial $\A$-graded ideal. Then we call the ideal
			\[J_\mon = \left\langle \bb{x}^{\bb{m}_1} - y_1\cdot s_1,...,\bb{x}^{\bb{m}_l} - y_l\cdot s_l \right\rangle\]
			from Construction \ref{cons:localincoherentequations} the \emph{universal family} of $\UM$ with \emph{defining ideal} $I_\mon$.
		\end{df}
		
		Remember that this time $I_\mon$ only describes the local ring of the toric Hilbert scheme at $\mon$, \textit{i.e.}
	  \[\IO_{\HA,[\mon]} \cong \field[Z]_{\left\langle Z\right\rangle}/F \cong \field[y_1,...,y_{p_\mon}]_{\left\langle y_1,...,y_{p_\mon}\right\rangle}/I_\mon.\]
	  Now we are interested in the components of $\UM$ that contain $\mon$. 
	  
	  So in our case we have $\UM = \Spec\left(\field[Z]/F\right)$, but we do not know $F$. Instead we have 
	  $\field[Z]_{\left\langle Z\right\rangle}/F\cong \field[\bb{y}]_{\left\langle \bb{y}\right\rangle}$. Recall
	  \[\renewcommand{\arraystretch}{1.35}\begin{array}{rl}
				G := & \left\langle \left.\bb{x}^{\bb{m}} - \xi_{\bb{m}}^{\bb{a}}\cdot s_{\bb{a}} 
												\,\right|\, \bb{a} \in \PD, \deg(\bb{x}^{\bb{m}})=\bb{a}, \bb{x}^{\bb{m}}\neq s_{\bb{a}}\right\rangle\subseteq \field[Z]\otimes_{\field}S \quad \textrm{and}\\
				F := & \sum_{\bb{a} \in \IN\A} \Fitt_0\left(\left( \field[Z][x_1,...,x_n]/G\right)_{\bb{a}}\right)\subseteq \field[Z].
			\end{array}\renewcommand{\arraystretch}{1}\]
		from Lemma \ref{lemma:UMring}. Note that each $y_i$ is one of the variables $\xi^{\bb{a}}_{\bb{m}} \in Z$. We denote the subset of these variables by $Z_{small}\subseteq Z$. 
		If we take some $\bb{x}^{\bb{m}} - \xi_{\bb{m}}^{\bb{a}}\cdot s_{\bb{a}}\in G$ with $\xi_{\bb{m}}^{\bb{a}} \notin Z_{small}$, then there is a reduction of 
		$\bb{x}^{\bb{m}} - \xi_{\bb{m}}^{\bb{a}}\cdot s_{\bb{a}}$ by $\overline{G_{\mon}}$ to
		\[\left(R\left(\bb{x}^{\bb{m}},\overline{G_{\mon}}\right)-\xi_{\bb{m}}^{\bb{a}}\right)\cdot s_{\bb{a}}\]
		as in Construction \ref{cons:localincoherentequations}, where $R\left(\bb{x}^{\bb{m}},\overline{G_{\mon}}\right)$ is a monomial in $Z_{small}$ which might be zero. Set
		\[Z_{red} := \left\{\left. \xi_{\bb{m}}^{\bb{a}} - R\left(\bb{x}^{\bb{m}},\overline{G_{\mon}}\right)\right| \xi_{\bb{m}}^{\bb{a}} \notin Z_{small}\right\}.\]
		Note that $Z_{red}\subseteq F$ and $\field[Z]/Z_{red}\cong \field[Z_{small}]\cong \field[y_1,...,y_{p_{\mon}}]$. Hence, there is an ideal $F'\subseteq \field[\bb{y}]$ 
		such that
		\[\field[Z]/F\cong \field[\bb{y}]/F'.\]
		But this means we get on the one hand
		\[\UM = \Spec\left(\field[\bb{y}]/F'\right)\]
		and on the other hand
		\begin{eqnarray*}
			\field[\bb{y}]_{\left\langle \bb{y} \right\rangle}/F' & \cong & \field[Z]_{\left\langle Z\right\rangle}/F\\
																														& \cong & \field[\bb{y}]_{\left\langle\bb{y}\right\rangle}/I_{\mon},
		\end{eqnarray*}
		where the isomorphism between $\field[\bb{y}]_{\left\langle \bb{y} \right\rangle}/F'$ and $\field[\bb{y}]_{\left\langle\bb{y}\right\rangle}/I_{\mon}$ is the identity. However, recall that
	  \[I_{\mon}=\left\langle r(u,v)\,\left|\, u,v\in \overline{G_\mon}\right.\right\rangle\subseteq \field[\bb{y}]_{\left\langle \bb{y}\right\rangle}\]
	  and that in fact $r(u,v)\in \field[\bb{y}]$ by construction. Thus, set the ideal
	  \[\widetilde{I_{\mon}}:=\left\langle r(u,v)\,\left|\, u,v\in \overline{G_\mon}\right.\right\rangle\subseteq \field[\bb{y}]\]
	  and the multiplicatively closed set $\mathcal{S} = \field[\bb{y}]\setminus \left\langle\bb{y}\right\rangle$. Then we have
	  \[\mathcal{S}^{-1}\widetilde{I_{\mon}}=I_{\mon}=\mathcal{S}^{-1}F'.\]
	  Now let
	  \[\widetilde{I_{\mon}}=\bigcap_{i=1}^k \Iq_i\]
	  be a minimal primary decomposition in $\field[\bb{y}]$ with prime ideals $\Ip_i=\sqrt{\Iq_i}$. Note that, although $\widetilde{I_{\mon}}$ is similar to $I_{\mon}$ in the 
	  coherent case before, the primary decomposition of it does not only give the components containing $\mon$. Therefore, we have to distinguish them further.
	  
	  Assume that $\Iq_i\cap \mathcal{S}=\emptyset$ for $i=1,...,m$ and  $\Iq_i\cap \mathcal{S}\neq \emptyset$ for $i=m+1,...,k$ for some $m$. Then by using 
	  \cite[Proposition 4.8 + 4.9]{AtiyahMacdonald:CommAlg} we have that
	  \[I_\mon = \bigcap_{i=1}^m \mathcal{S}^{-1}\Iq_i\quad \textrm{and}\quad I_{\mon} \cap \field[\bb{y}]=\bigcap_{i=1}^m \Iq_i\]
	  are minimal primary decompositions. On the other hand,
	  \[I_\mon\cap \field[\bb{y}] = \bigcup_{s\in \mathcal{S}} \left(\widetilde{I_{\mon}}:s\right) = \bigcup_{s\in \mathcal{S}} \left(F':s\right)\]
	  is the saturation of $I_\mon$ considered as an ideal in $\field[\bb{y}]$ with respect to $\field[\bb{y}]\setminus \left\langle\bb{y}\right\rangle$ and moreover the saturation of $F'$. 
	  But the latter are the functions that do not vanish on the point $\left(\bb{0}\right)$, the point that corresponds to $\mon$. Hence, 
	  $V\left(I_\mon\cap \field[\bb{y}]\right) \subseteq \UM$ is the intersection of $\UM$ with all components of $\HA$, that contain $M$. 
	  Thus, $V\left(\Iq_i \cap \field[\bb{y}]\right)\subseteq \UM$ is isomorphic to an irreducible 
	  subset of $\UM$ containing $\mon$ and all $(\Iq_i \cap \field[\bb{y}])$ give exactly the reduced components of $\HA$, that contain $\mon$, 
	  intersected with $\UM$. Note that again there might be embedded components.
	  
	  These primary ideals give the following description of $\UM$.
	  
	  \begin{prop}\label{prop:incoherentprimaryideals}
	  	Let $\mon$ be a non-coherent monomial $\A$-graded ideal with universal family $J_\mon'$ and corresponding defining ideal $I_\mon'$, both after removing 
	  	redundant variables. Let $I_\mon' = \Iq_1 \cap ... \cap \Iq_k$ be the primary decomposition in $\field[y_i \,|\, i \in \rv]$. Then 
	  	$\overline{V(\Iq_j)}\subseteq \HA$ is an irreducible component containing $\mon$ if and only if none of the generators of $\Iq_j$ is 
	  	a unit in the localisation $\field[y_i \,|\, i \in \rv]_{\left\langle y_i|i\in\rv\right\rangle}$.
	  \end{prop}
	  
	  \begin{proof}
	  	First of all note that removing redundant variables still maps $I_\mon$ to an isomorphic description in the local ring since
	  	\[y_i - \prod_{j\neq i} y_j^{b_j}\]
	  	is not a unit in $\field[\bb{y}]_{\left\langle\bb{y}\right\rangle}$.
	  	Let $\Iq_i$ be one of the primary ideals in a minimal primary decomposition of $I_\mon$ in $\field[\bb{y}]$ with prime ideal 
	  	$\Ip_j = \sqrt{\Iq_j}$ and denote the multiplicatively closed set $\mathcal{S} := \field[y_i\,|\, i \in \rv]\setminus \left\langle y_i|i\in\rv\right\rangle$. Then $V(\Iq_j)$ 
	  	is the intersection of a reduced irreducible component containing $\mon$ with $\UM$ if and only if $\Ip_j$ is an associated prime of 
	  	$\bigcup_{s\in \mathcal{S}} \left(I_\mon:s\right)$. But by using the above, 
	  	$\Ip_j$ is an associated prime of $\bigcup_{s\in \mathcal{S}} \left(I_\mon:s\right)$ exactly if none of the generators of $\Iq_j$ is a 
	  	unit in $\mathcal{S}$. Again, by Lemma \ref{lem:localChartcontainsTorus} the closure $\overline{V(\Iq_j)}$ is an irreducible component of $\HA$.
	  \end{proof}
	  
	  \begin{remnn}
	  	Note that all components of $\UM$ containing $\mon$ are in fact non-coherent and therefore all primary ideals contain a monomial generator.
	  \end{remnn}
	  
	  From now on the construction is the same for coherent and non-coherent monomial ideals, since the primary ideals in Propositions 
	  \ref{prop:coherentprimaryideals} and \ref{prop:incoherentprimaryideals} giving non-coherent components have exactly the same properties.
	  We want to construct the polytope that defines the reduced underlying structure of a non-coherent component, so we fix one of these primary ideals 
	  $\Iq_i$ and take the radical $\Ip_i := \sqrt{\Iq_i}$. Then $V(\Ip_i)$ is an affine open chart of the underlying reduced scheme of the non-coherent 
	  component given by $\Iq_i$. To be more precise, it is isomorphic to an affine open chart, where the isomorphism is given by the universal 
	  family $J_\mon'$. In other words, if the $\bb{y}$'s in $J_\mon'$ are considered as coefficients, then those coefficients satisfying $\Ip_i$ give exactly the 
	  $\A$-graded ideals that correspond to the points of that component. 
	  
	  \begin{df}\label{df:ComponentofAnAssociatedPrime}
	  	Let $\Ip$ be an associated prime of $I_\mon'$ with corresponding primary ideal $\Iq$ and irreducible component $\overline{V(\Iq)}\subseteq \HA$ 
	  	as in Proposition \ref{prop:coherentprimaryideals} or \ref{prop:incoherentprimaryideals}. Then 
	  	we denote the reduced scheme of the corresponding component by 
	  	\[V_{\Ip}:=\left(\overline{V(\Iq)}\right)_{\textrm{red}}\subset \left(\HA\right)_{\textrm{red}}.\]
	  \end{df}
	  
	  \begin{rem}
	  	Since $\Ip =\sqrt{\Iq}$ we have
	  	\[V_{\Ip} = \overline{V(\Ip)} \subset \left(\HA\right)_{\textrm{red}}.\]
	  \end{rem}
	  
	  We now give a construction via torus invariant isomorphisms to get a universal 
	  family $J_\mon(\Ip)$ that gives an open affine chart of the component for all values of the remaining $\bb{y}$-variables. This means we will perform 
	  a change of coordinates on the $\bb{y}$'s in $J_{\mon}'$ and $\Ip$, that makes $\Ip$ trivial. Since the solutions of $\Ip$ give all $\A$-graded 
	  ideals in the torus of that component we will get a new universal family where every set of values for $\bb{y}$ gives a point on that component.
	  
	  Since $\Ip$ is a binomial prime ideal, a minimal generating set is of the form
	  \[\Ip = \left\langle y_i,\bb{y}^{\bb{b}^+}-\bb{y}^{\bb{b}^-}\right\rangle\]
	  for some $i \in \rv$ and some $\bb{b}^+,\bb{b}^- \in \IN^{\rv}$. Then the first step is to remove the $y_i$'s in $\Ip$ and $J_{\mon}'$ by just setting 
	  them to zero. 
	  
	  \begin{cons}[\textbf{Removing single variables}]\label{cons:removeSingleVariables}
	  	Let $\mon$ be a monomial $\A$-graded ideal and $\Ip$ an associated prime of $I_{\mon}'$. Denote by $\rv'$ the indices of the variables, that are not contained in 
	  	$\Ip$, and by $J_{\Ip}$ the set of all $j\in \{1,...,p_{\mon}\}$ such that $y_i\not|\; p_j(\bb{y})$ for all $y_i \in \Ip$, \textit{i.e.} all indices where $p_j(\bb{y})$ 
	  	remains unchanged and is not set to zero.
	  	Then we remove the single variables by applying the map
	  	\[\Psi : \field[y_i\,|\, i \in \rv] \rightarrow \field[y_i\,|\, i \in \rv'],\, y_ j \mapsto \left\{\begin{array}{cc} 
																					y_j & \textrm{if } j \in \rv'\\ 
																					0 & \textrm{if } j \notin \rv'
																						\end{array}\right.\]
			to $\Ip$, where we get
			\[\Ip' := \Psi(\Ip) = \left\langle \bb{y}^{\bb{b}^+}-\bb{y}^{\bb{b}^-}\right\rangle,\]
	  	and by applying $\Id_{\field[\bb{x}]} \otimes \Psi$ to $J_{\mon}'$, where we get 
	  	\[J_{\mon}'' := \Id_{\field[\bb{x}]} \otimes \Psi(J_{\mon}') = \left\langle \left.\bb{x}^{\bb{m}_j}-\bb{y}^{\bb{b}_j}\cdot s_j\,\right|\, j \in J_{\Ip}\right\rangle + 
	  							\left\langle \left.\bb{x}^{\bb{m}_j}\,\right|\,j\in J_{\Ip}\right\rangle\]
	  	in $\field[\bb{x},y_i\,|\, i\in \rv']$, for some $\bb{b}_j \in \IN^{\rv'}, j \notin J_{\Ip}$.
	  \end{cons}
	  
	  \begin{rem}\label{rem:localaffinechart}
	  	For the open affine chart of $V_{\Ip}$ containing $\mon$ given by $\UM$ we get the isomorphism $V_{\Ip}\cap (\UM)_{\textrm{red}} \cong \Spec\left(\field[y_i\,|\, i \in \rv']/\Ip'\right)$, 
	  	where $(\UM)_{\textrm{red}}$ is the underlying reduced scheme of $\UM$.
	  \end{rem}
	  
	  Because $\Ip'$ is prime, if looking at a generator $\bb{y}^{\bb{b}^+}-\bb{y}^{\bb{b}^-}$, the difference of the exponent vectors $\bb{b}:=\bb{b}^+ - \bb{b}^-$ is coprime. 
	  Hence, there is an isomorphism $A \in \Gl(\rv',\IZ)$ such that $A\cdot\bb{b} = e_1$, the first vector of the canonical basis. This is equivalent to a 
	  torus invariant isomorphism
	  \[\Phi_A : \field\hspace{-3pt}\left[\left.y_i^{\pm 1} \,\right|\, i \in \rv'\right] \rightarrow \field\left[\left.y_i'^{\pm 1} \,\right|\, i \in \rv'\right],\quad y_i \mapsto \bb{y}'^{A_i},\]
	  where $A_i$ denotes the $i$-th column of $A$. This means that on the spectrum of these rings $\Phi_A$ gives an isomorphism on their tori. Using $\Phi_A$, we can 
	  map $\Ip'$ to some prime ideal $\Phi_A(\Ip')$ in $\field[y_i' \,|\, i \in \rv']$ by sending the binomial $\bb{y}^{\bb{b}^+}-\bb{y}^{\bb{b}^-}$ with 
	  $\bb{b}=\bb{b}^+ - \bb{b}^-$ to the binomial
	  \[ \bb{y}'^{A(\bb{b})^+}-\bb{y}'^{A(\bb{b})^-},\]
	  which differs only by a unit from
	  \[ \bb{y}'^{(A\bb{b})^+}-\bb{y}'^{(A\bb{b})^-},\]
	  where $A\bb{b}=(A\bb{b})^+ - (A\bb{b})^-$ is the unique decomposition into two positive vectors. If we extend $\Phi_A$ by the identity on the $x_i$, we can apply 
	  it to the altered universal family $J_\mon''=\left\langle \left.\bb{x}^{\bb{m}_j}- \bb{y}^{\bb{b}_j}\cdot s_j\,\right|\, j \in J_{\Ip}\right\rangle 
	  					+ \left\langle \left.\bb{x}^{\bb{m}_j} \,\right|\, j \notin J_{\Ip}\right\rangle$ to get
	  \[\Phi_A(J_\mon''):=\left\langle \left.\bb{x}^{\bb{m}_j}- \bb{y}'^{A\bb{b}_j}\cdot s_j\,\right|\, j \in J_{\Ip}\right\rangle + 
	  				\left\langle \left.\bb{x}^{\bb{m}_j} \,\right|\, j \notin J_{\Ip}\right\rangle\]
	  in $\field[\bb{x},\bb{y}'^{\pm 1}]$. Here the $\bb{y}'$ terms have become Laurent monomials, as they might have negative exponents.
	  
	  \begin{cons}[\textbf{Change of coordinates}]\label{cons:Isomorphism}
	  	Let $\bb{y}^{\bb{b}^+}-\bb{y}^{\bb{b}^-} \in \Ip'$ be an element of a minimal generating set. Fix a matrix $A \in \Gl(\rv',\IZ)$ with torus invariant morphism 
	  	\[\Phi_A : \field\left[\left.y_i^{\pm 1} \,\right|\, i \in \rv'\right] \rightarrow \field\left[\left.y_i'^{\pm 1} \,\right|\, i \in \rv'\right],\quad y_i \mapsto \bb{y}'^{A_i},\]
	  	such that $\Phi_A(\bb{y}^{\bb{b}^+}-\bb{y}^{\bb{b}^-})=y_1'-1$. Then compute the new universal family
	  	\[\Phi_A(J_\mon''):=\left\langle \left.\bb{x}^{\bb{m}_j}- \bb{y}'^{A\bb{b}_j}\cdot s_j\,\right|\, j \in J_{\Ip}\right\rangle + 
	  				\left\langle\left. \bb{x}^{\bb{m}_j} \,\right|\, j \notin J_{\Ip}\right\rangle\]
	  	where we set $y_1'$ to be $1$ and the new prime ideal
	  	\[\Phi_A(\Ip')=\left\langle\left. \bb{y}'^{(A\bb{b})^+}-\bb{y}'^{(A\bb{b})^-}\,\right|\, \bb{y}^{\bb{b}^+}-\bb{y}^{\bb{b}^-}\in\Ip' \right\rangle\]
	  	from which we also remove $y_1'-1$ since it has become zero.
	  \end{cons}
	  
	  \begin{lem}\label{lem:sameComponent}
	  	Let $J_\mon = \left\langle\left. \bb{x}^{\bb{m}_j}- \bb{y}^{\bb{b}_j}\cdot s_j\,\right|\, j \in J_{\Ip}\right\rangle 
	  					+ \left\langle\left. \bb{x}^{\bb{m}_j} \,\right|\, j \notin J_{\Ip}\right\rangle$ be a universal family of a local chart of $\HA$ and 
	  	$\Ip \subseteq \field[y_i \,|\, i \in \rv]$ be a binomial prime ideal with no monomial generators that gives a reduced irreducible component on this chart. 
	  	For a generator $\bb{y}^{\bb{b}^+}-\bb{y}^{\bb{b}^-}$ of $\Ip$ choose an isomorphism $\Phi_A$ as above and set the universal family 
	  	$\Phi_A(J_\mon)$ in $\field[\bb{x},\bb{y}'^{\pm 1}]$ and the prime ideal $\Phi_A(\Ip) \in \field[y_i' \,|\, i \in \rv]$ as before. Then the prime ideal 
	  	$\Phi_A(\Ip)$ gives the intersection of the same irreducible component with its ambient torus.
	  \end{lem}
	  
	  \begin{proof}
	  	Consider $\Phi_A$ on the Laurent polynomials of both rings:
	  	\[\Phi_A : \field\left[\left.y_i^{\pm 1} \,\right|\, i \in \rv\right] \rightarrow \field\left[\left.y_i'^{\pm 1} \,\right|\, i \in \rv\right],\quad y_i \mapsto \bb{y}'^{A_i}\]
	  	This induces the second isomorphism of 
	  	\[V(\Ip)\cong \field\left[\left.y_i^{\pm 1}\,\right|\, i \in \rv\right]/\Ip \cong \field\left[\left.y_i'^{\pm 1} \,\right|\, i \in \rv\right]/\Phi_A(\Ip).\]
	  	This means that the points $\left(\lambda_i\right)_{i\in\rv}$ in $V\left(\Phi_A(\Ip)\right)$ substituted into the $\bb{y}$ variables in $\Phi_A\left(J_{\mon}\right)$ 
	  	parametrise $V(\Ip)\cap T_{\Ip}$. Hence, also $V\left(\Phi_A(\Ip)\right)$ gives the ambient torus of $V_{\Ip}$.
	  \end{proof}
	  
	  The advantage of Lemma \ref{lem:sameComponent} is that the binomial $\bb{y}^{\bb{b}^+}-\bb{y}^{\bb{b}^-}$, that we have used to get $A$, is sent to 
	  $y_1'-1$ under $\Phi_A$. Hence, we can substitute $1$ for $y_1'$ in $\Phi_A(\Ip)$ and $\Phi_A(J_\mon)$ and by this remove one more variable and one generator 
	  of $\Phi_A(\Ip)$. The resulting prime ideal and universal family again satisfy the conditions for Lemma \ref{lem:sameComponent} and thus we can repeat this 
	  reduction until $\Ip$ has become the zero ideal. Thus, we can remove $\Ip$ with the following construction.
	  
	  \begin{cons}[\textbf{Computing the universal family}]\label{cons:universalFamilyOfComponent}
	  	Let $\mon$ be a monomial $\A$-graded ideal. Compute the universal family $J_\mon$ and the defining ideal $I_\mon$ as in Proposition \ref{cons:localcoherentequations} 
	  	or \ref{cons:localincoherentequations}, if $\mon$ is coherent or non-coherent, respectively. Then reduce the redundant variables in $J_\mon$ and $I_\mon$ according to 
	  	Construction \ref{cons:reduceRedundantVariables} to $J_\mon'$ and $I_\mon'$. If $\mon$ is coherent use Proposition \ref{prop:coherentprimaryideals} and if $\mon$ is 
	  	non-coherent use Proposition \ref{prop:incoherentprimaryideals} to determine the primary ideals $\Iq_1,...,\Iq_m$ that determine the non-coherent components 
	  	containing $\mon$. Let $\Ip =\sqrt{\Iq_i}$ be one of the associated primes. Then use Construction \ref{cons:removeSingleVariables} to remove the single 
	  	variables given by $\Ip$ from $J_\mon'$ and $\Ip$ to get $J_\mon''$ and $\Ip'$, respectively. Pick a minimal generator $\bb{y}^{\bb{b}_1^+}-\bb{y}^{\bb{b}_1^-}$ of $\Ip'$ 
	  	and use the corresponding isomorphism $\Phi_{A_1}$ as in Construction \ref{cons:Isomorphism} to get $\Phi_{A_1}(\Ip')$ and $\Phi_{A_1}(J_\mon'')$. Repeat this until the 
	  	image of the prime ideal under the repeated isomorphisms is $\Phi_{A_k}(...(\Phi_{A_1}(\Ip')))=(0)$. Denote the ideal resulting from applying 
	  	$\Phi_{A_k}\circ ...\circ \Phi_{A_1}$ to $J_\mon''$ by $J_\mon(\Ip)$.
	  \end{cons}
	  
	  \begin{rem}
	  	The $k$ steps in Construction \ref{cons:universalFamilyOfComponent} can also be done in one step by using one isomorphism over $\IZ$ that maps to the torus. We have 
	  	shown the removal of $\Ip$ step by step just for lucidity. Although, when implementing this construction one should use the single isomorphism over $\IZ$ to the torus. 
	  	Furthermore, one can combine algorithmically Construction \ref{cons:removeSingleVariables} and the repeated steps of Construction \ref{cons:Isomorphism} into one single 
	  	morphism to the torus over $\IZ$.
	  \end{rem}
	  
	  \begin{df}\label{df:universalFamilyOfAComponent}
	  	Let $\mon$ be a monomial $\A$-graded ideal, $J_\mon$ the universal family of $\UM$ with defining ideal $I_\mon$, $\Ip$ a prime ideal as in 
	  	Proposition \ref{prop:coherentprimaryideals} or \ref{prop:incoherentprimaryideals} defining the underlying reduced scheme $V_{\Ip}$ of a non-coherent 
	  	irreducible component containing $\mon$. Then we call the ideal resulting from removing the single variables in $\Ip$ as in Construction 
	  	\ref{cons:removeSingleVariables} and the reduction of $J_\mon$ by every generator of $\Ip$ as in Construction \ref{cons:universalFamilyOfComponent}
	  	\[J_\mon(\Ip)=\left\langle\left. \bb{x}^{\bb{m}_j}- \bb{y}^{\bb{b}_j}\cdot s_j\,\right|\, j\in J_{\Ip} \right\rangle +
	  			\left\langle\left. \bb{x}^{\bb{m}_j}\,\right|\, j \notin J_{\Ip}\right\rangle\]
	  	in $\field\left[\left.\bb{x},y_i^{\pm 1} \,\right|\, i \in \rv(\Ip)\right]$ \emph{the universal family of the component $V_{\Ip}$}, where $\rv(\Ip)$ denotes the remaining variables 
	  	and $\bb{b}_j \in \IZ^{\# \rv(\Ip)}$ are the resulting exponents.
	  \end{df}
	  
	  \begin{thm}\label{thm:universalfamilyofcomponent}
	  	Let $J_\mon$ be a universal family with a prime ideal $\Ip$, together giving an affine chart of a reduced irreducible component $V_{\Ip}$ of the toric Hilbert scheme.	Let 
	  	$J_\mon(\Ip)$ be the universal family of this component. Then $(\field^*)^{\# \rv(\Ip)}=\Spec\left(\field\left[\left.y_i^{\pm 1} \,\right|\, i \in \rv(\Ip)\right]\right)$ 
	  	is isomorphic to the reduced irreducible component $V_{\Ip}$ intersected with its ambient torus by substituting these points for the $\bb{y}$ variables in $J_\mon(\Ip)$. 
	  	To be precise, the closed points of this irreducible component of the toric Hilbert scheme intersected with its ambient torus are exactly those $\A$-graded ideals that are 
	  	given by substituting a point $(\lambda_i)_{i\in \rv(\Ip)} \in (\field^*)^{\# \rv(\Ip)}$ for the $\bb{y}$ variables in $J_\mon(\Ip)$.
	  \end{thm}
	  
	  \begin{proof}
	  	The theorem follows directly from Lemma \ref{lem:sameComponent}. Denote by $A_i$ the matrix of the $i$-th reduction and let $V_{\Ip}^T$ be the intersection 
	  	of the reduced irreducible component, given by $\Ip$, with its torus. Now we use the lemma at every step of the reduction to get the isomorphism between $V_{\Ip}^T$ and 
	  	$V(\Phi_{A_k}(...(\Phi_{A_1}(\Ip)))) \cap T$ via $\Phi_{A_k}(...(\Phi_{A_1}(J_\mon)))$. A minimal generating set of $\Ip$ is finite, so after a finite number $h$ of 
	  	reduction steps we get $\Phi_{A_h}(...(\Phi_{A_1}(\Ip)))=(0)$ because we have removed all generators, thus $(\field^*)^{\# \rv(\Ip)}=V(0)$ is isomorphic to $V_{\Ip}^T$ via 
	  	$J_\mon(\Ip)=\Phi_{A_h}(...(\Phi_{A_1}(J_\mon)))$.
	  \end{proof}
	  
	  \begin{cor}\label{rem:dimensionofcomponent}
	  	The ambient torus of a non-coherent irreducible component of $\HA$ is given by one universal family 
	  	$J_\mon(\Ip)$ in $\field\left[\left.\bb{x},y_i^{\pm} \,\right|\, i \in \rv(\Ip)\right]$. Hence, the dimension of this component is $\# \rv(\Ip)$. This 
	  	means that the dimension of every non-coherent component is bounded by the number of elements in the Graver basis, since $\# \rv(\Ip)$ is bounded by the 
	  	number of generators of $J_\mon$ and all $\bb{x}^{\bb{m}_j}-s_j$ are Graver.
	  	
	  	Furthermore, since $J_\mon(\Ip)$ gives the ambient torus of the irreducible component, the closure of the torus is the whole component.\qed
	  \end{cor}
	  
	  \begin{rem}
	  	The ambient torus of a reduced non-coherent component $V_{\Ip}$ is given by 
	  	$T_{\Ip}:=\Spec\left(\field\left[\left.y_i^{\pm 1}\,\right|\, i \in \rv(\Ip)\right]\right)$, \textit{i.e.} the points of $T_{\Ip}$ 
	  	correspond to the points in the ambient torus of $V_{\Ip}$ via $J_\mon(\Ip)$. We refer to $T_{\Ip}$ as \emph{the ambient torus of the non-coherent component $V_{\Ip}$}. Note that 
	  	the ambient torus of the coherent component is the $n$-torus $T=\Spec(\field[\bb{x}^{\pm 1}])$ to which $T_{\Ip}$ is the analog for a non-coherent component.
	  \end{rem}
	  
	  We now give a slight variation of Theorem \ref{thm:universalfamilyofcomponent} that avoids Laurent monomials.
	  
	  \begin{cor}
	  	Let $J_\mon(\Ip) = \left\langle\left. \bb{x}^{\bb{m}_j}- \bb{y}^{\bb{b}_j}\cdot s_j\,\right|\, j\in J_{\Ip} \right\rangle + 
	  				\left\langle\left. \bb{x}^{\bb{m}_j} \,\right|\, j \notin J_{\Ip} \right\rangle$ be the universal family 
	  	of $V_{\Ip}$ in $\field\left[\left.\bb{x},y_i^{\pm 1} \,\right|\, i \in \rv(\Ip)\right]$.	Then 
	  	\[J_\mon(\Ip)' = \left\langle\left. \bb{y}^{\bb{b}_j^-}\cdot\bb{x}^{\bb{m}_j}- \bb{y}^{\bb{b}_j^+}\cdot s_j\,\right|\, j\in J_{\Ip} \right\rangle + 
	  			\left\langle\left. \bb{x}^{\bb{m}_j}\,\right|\, j\notin J_{\Ip}\right\rangle\]
	  	in $\field\left[\left.\bb{x},y_i^{\pm} \,\right|\, i \in \rv(\Ip)\right]$	is also a universal family for $V_{\Ip}$ giving an isomorphism between $(\field^*)^{\# \rv(\Ip)}$ and $V_{\Ip}$ 
	  	intersected with its ambient torus, where $\bb{b}_j = \bb{b}_j^+ - \bb{b}_j^-$ is the unique decomposition into two positive vectors. 
	  \end{cor}
	  
	  \begin{proof}
	  	Just note that it is an isomorphism of tori. Thus we do not change the isomorphism by multiplying the $j$-th generator of $J_\mon(\Ip)$ with 
	  	$\bb{y}^{\bb{b}_j^-}$, because this is just multiplication with a unit.
	  \end{proof}
	  
	  The universal family $J_\mon(\Ip)$ has the advantage that it defines the ambient torus of the non-coherent irreducible component on its own, which means we do not need any equations 
	  on the coefficients in $\bb{y}$ anymore. 
	  
	  We now apply Construction \ref{cons:universalFamilyOfComponent} to a monomial $\A$-graded ideal where all steps of the construction of $J_\mon(\Ip)$ have to be done.
	  
	  \begin{ex}\label{ex:The2by6example}
	  	 Let $\A=\left\{\vect{0}{6},\vect{2}{4},\vect{3}{0},\vect{3}{7},\vect{4}{2},\vect{6}{1}\right\}\subset \IZ^2$. Then the toric ideal is
	  	 \begin{eqnarray*}
	  	 	I_{\A} & = & \left\langle bc^2-e^2, ac^2-be, b^2-ae, cd-af, c^8e^3-f^6, c^3e^6-df^5, be^7-d^2f^4, \right.\\
	  	 	       &   & \left.a^2ce^6-d^3f^3, a^3e^6-d^4f^2, a^4bce^4-d^5f, d^6-a^5be^4\right\rangle.
	  	 \end{eqnarray*}
	  	 The Graver basis has $381$ elements and there are $9588$ monomial $\A$-graded ideals, which were found by using the algorithm in 
	  	 \cite[Section 1]{StillmanSturmfelsThomas:Algorithms}. We choose the non-coherent monomial $\A$-graded ideal 
	  	 \begin{eqnarray*}
	  	 		\mon & = & \left\langle bc^2, ae, ac^2, cd, abcf, a^2cf, b^3cf, a^2bf^2, a^3f^2, ab^3f^2, c^8e^3, \right.\\
	  	 	  	&   & \left. b^5f^2, df^5, bf^6, d^2f^4, af^6, b^4ce^4, d^4f^2, b^9c, ad^3f^3, ab^9, d^6e\right\rangle
	  	 	\end{eqnarray*}
	  	 	which has $22$ generators. Therefore, the universal family $J_\mon$ of this monomial ideal is in $\field[a,b,c,d,e,f,y_1,...,y_{22}]$, 
	  	 	and the defining ideal $I_\mon$ lies in $\field[y_1,...,y_{22}]$ and has $40$ generators. The equations give $14$ redundant variables 
	  	 	(all except $y_4,y_{11},y_{12},y_{14},y_{17},y_{20},y_{21}$, and $y_{22}$), so if we remove them we get 
	  	 	\begin{eqnarray*}
	  	 		J_\mon'&=&\left\langle bc^2-y_{11}y_{14}e^2, ae-y_{21}y_{22}b^2, ac^2-y_4^2y_{12}y_{21}^3y_{22}^3be, cd-y_4af,\right.\\
	  	 				& &\left. abcf-y_4y_{12}y_{21}^2y_{22}^2de^2,a^2cf-y_4y_{12}y_{21}^3y_{22}^3bde,b^3cf-y_4y_{12}y_{21}y_{22}de^3,\right.\\
	  	 		    & &\left. a^2bf^2-y_{12}y_{21}^2y_{22}^2d^2e^2, a^3f^2-y_{12}y_{21}^3y_{22}^3bd^2e, ab^3f^2-y_{12}y_{21}y_{22}d^2e^3,\right.\\
	  	 		    & &\left. c^8e^3-y_{11}f^6, b^5f^2-y_{12}d^2e^4, df^5-y_4^3y_{12}y_{14}y_{21}^3y_{22}^3c^3e^6, bf^6-y_{14}c^6e^5,\right.\\
	  	 		    & &\left. d^2f^4-y_{12}^2y_{22}be^7, af^6-y_4^2y_{12}y_{14}y_{21}^3y_{22}^3c^4e^6,b^4ce^4-y_{17}d^3f^3,\right.\\
	  	 		    & &\left. d^4f^2-y_{12}y_{22}b^6e^3, b^9c-y_4y_{21}d^5f, ad^3f^3-y_{20}b^6ce^3, \right.\\
	  	 		    & &\left. ab^9-y_{21}d^6, d^6e-y_{22}b^{11}\right\rangle
	  	 	\end{eqnarray*}
	  	 	and the defining ideal
	  	 	\begin{eqnarray*}
	  	 		I_\mon' & = & \left\langle y_{17}y_{20}-y_{21}y_{22},y_4y_{20}-y_{12}y_{22},y_{12}y_{17}-y_4y_{21},y_4^2y_{12}y_{21}^2y_{22}^2-y_{11}y_{14},\right.\\
	  	 		     &   & \left. y_4y_{11}^2y_{14}^3y_{17}y_{21}y_{22}-y_{12},y_{11}^3y_{14}^4y_{21}y_{22}-y_{12}^3y_{21}y_{12}^2\right\rangle
	  	 	\end{eqnarray*}  	 		     
	  	 	for which we can construct the primary decomposition
	  	 	\begin{eqnarray*}
	  	 		I_\mon' &   =  & \left\langle y_{17}y_{20}-y_{21}y_{22},y_4y_{20}-y_{12}y_{22},y_{12}y_{17}-y_4y_{21},y_4^2y_{12}y_{21}^2y_{22}^2-y_{11}y_{14},\right.\\
	  	 				 &			& \left. y_{11}^2y_{14}^3y_{17}^2y_{22}-1,y_{11}^3y_{14}^4y_{21}y_{22}-y_{12}^3y_{21}y_{22}^2\right\rangle \\
	  	 		     & \cap & \left\langle y_4,y_{11},y_{12},y_{17}y_{20}-y_{21}y_{22}\right\rangle\\
	  	 		     & \cap & \left\langle y_{11},y_{12},y_{20},y_{21}\right\rangle\\
	  	 		     & \cap & \left\langle y_4,y_{12},y_{14},y_{17}y_{20}-y_{21}y_{22}\right\rangle.
	  	 	\end{eqnarray*}
	  	 	The first primary ideal contains an element of $\field[\bb{y}]\setminus \left\langle\bb{y}\right\rangle$, so it 
	  	 	does not define a non-coherent component containing $\mon$. Hence, $\mon$ lies on $3$ non-coherent components 
	  	 	and all three of them are reduced. Therefore, 
	  	 	$\Ip = \left\langle y_4,y_{11},y_{12},y_{17}y_{20}-y_{21}y_{22}\right\rangle$ defines an affine chart of a reduced 
	  	 	irreducible component containing $\mon$. Now we apply Construction \ref{cons:removeSingleVariables} to remove the 
	  	 	single variables in $\Ip$ which gives $\Ip'=\left\langle y_{17}y_{20}-y_{21}y_{22}\right\rangle$ and the new universal family
	  	 	\begin{eqnarray*}
	  	 		J_\mon''&=&\left\langle ae-y_{21}y_{22}b^2, bf^6-y_{14}c^6e^5, b^4ce^4-y_{17}d^3f^3, ad^3f^3-y_{20}b^6ce^3,\right.\\
	  	 		     & &\left. ab^9-y_{21}d^6, d^6e-y_{22}b^{11} \right\rangle + \left\langle bc^2, ac^2, cd, abcf, a^2cf, b^3cf,\right.\\
	  	 		     & &\left. a^2bf^2, a^3f^2, ab^3f^2, c^8e^3, b^5f^2, df^5, d^2f^4, af^6,d^4f^2, b^9c\right\rangle.
	  	 	\end{eqnarray*}
	  	 	There is one more binomial generator $y_{17}y_{20}-y_{21}y_{22}$ in $\Ip'$ left, which has the exponent vector $\left(0,1,1,-1,-1\right)^t$ in the remaining 
	  	 	variables $y_{14},y_{17},y_{20},y_{21},y_{22}$. Hence, we have to apply one isomorphism from Construction \ref{cons:Isomorphism} to remove that binomial. Our choice for this is 
	  	 	\[A := \left(\begin{array}{ccccc}
	  	 								0& 1&0&0&0\\
	  	 								1& 0&0&0&0\\
	  	 								0&-1&1&0&0\\
	  	 								0& 1&0&1&0\\
	  	 								0& 1&0&0&1	
	  	 							\end{array}\right)\in \Gl(5,\IZ).\]
	  	 	This means we get the isomorphism 
	  	 	\[\Phi_A : \field[y_{14}^{\pm 1},y_{17}^{\pm 1},y_{20}^{\pm 1},y_{21}^{\pm 1},y_{22}^{\pm 1}] \rightarrow \field[y_0^{\pm 1},...,y_4^{\pm 1}],\]
	  	 	that maps
	  	 	\[y_{14}\mapsto y_1,\;\; y_{17}\mapsto \frac{y_0y_3y_4}{y_2},\;\; y_{20} \mapsto y_2,\;\; y_{21} \mapsto y_3,\;\; \textrm{ and } y_{22}\mapsto y_4.\]
	  	 	Hence, $\Phi_A(\Ip')=y_0-1$ and the universal family is mapped to
	  	 	\begin{eqnarray*}
	  	 		\Phi_A(J_\mon'')&=&\left\langle ae-y_3y_4b^2, bf^6-y_1c^6e^5, b^4ce^4-\frac{y_0y_3y_4}{y_2}d^3f^3,\right.\\
	  	 								 & &\left. ad^3f^3-y_2b^6ce^3, ab^9-y_3d^6, d^6e-y_4b^{11}\right\rangle + \\
	  	 								 & &\left\langle bc^2, ac^2, cd, abcf, a^2cf, b^3cf, a^2bf^2, a^3f^2,\right.\\
	  	 		    				 & &\left. ab^3f^2, c^8e^3, b^5f^2, df^5, d^2f^4, af^6,d^4f^2, b^9c\right\rangle.
	  	 	\end{eqnarray*}
	  	 	But now we have to set $y_0$ to $1$ so that the universal family of the component $V_{\Ip}$ is 
	  	 	\begin{eqnarray*}
	  	 		J_\mon(\Ip) & = & \left\langle ae-y_3y_4b^2, bf^6-y_1c^6e^5, b^4ce^4-\frac{y_3y_4}{y_2}d^3f^3,\right.\\
	  	 						 &   & \left. ad^3f^3-y_2b^6ce^3, ab^9-y_3d^6, d^6e-y_4b^{11}\right\rangle + \\
	  	 						 &	 & \left\langle bc^2, ac^2, cd, abcf, a^2cf, b^3cf, a^2bf^2, a^3f^2,\right.\\
	  	 		    		 &   & \left. ab^3f^2, c^8e^3, b^5f^2, df^5, d^2f^4, af^6,d^4f^2, b^9c\right\rangle
	  	 	\end{eqnarray*}
	  	 	in $\field[a,b,c,d,e,f,y_1,y_2,y_3,y_4]$.\EX	 	
	  \end{ex}

\section{Isomorphic Universal Families}     %

		So far, we have constructed universal families $J_\mon(\Ip)$ for every monomial $\A$-graded ideal $\mon$ which gives
		the ambient torus of some reduced non-coherent component $V_{\Ip}$ containing $\mon$. Thus, if we want to describe all non-coherent components, 
		we have to compute all universal families for each monomial $\A$-graded ideal. But this means that we would construct for one 
		non-coherent irreducible component different universal families, one for each monomial $\A$-graded ideal 
		in that component. Hence, we have to find a method to identify two universal families defining the same non-coherent component.
		
		Consider two monomial $\A$-graded ideals $\mon_1,\mon_2$ with two prime ideals $\Ip_1,\Ip_2$ giving reduced irreducible components $V_{\Ip_1},V_{\Ip_2}$ 
		of $\HA$ that contain $\mon_1$ and $\mon_2$, respectively. Then we have the two universal families
		\begin{equation}\label{eqn:twouniversalfamilies}
			\renewcommand{\arraystretch}{1.15}\begin{array}{rcl}
				J_1 := & J_{\mon_1}(\Ip_1)= &  \left\langle\left. \bb{x}^{\bb{m}_j}- \bb{y}^{\bb{b}_j}\cdot \bb{x}^{\bb{n}_j}\,\right|\, j\in J_{\Ip_1} \right\rangle 
						+ \left\langle\left. \bb{x}^{\bb{m}_j} \,\right|\, j \notin J_{\Ip_1} \right\rangle \textrm{ and}\\
				J_2 := & J_{\mon_2}(\Ip_2)= &  \left\langle\left. \bb{x}^{\bb{u}_j}- \bb{y}'^{\bb{c}_j}\cdot \bb{x}^{\bb{v}_j}\,\right|\, j\in J_{\Ip_2} \right\rangle 
						+ \left\langle\left. \bb{x}^{\bb{u}_j} \,\right|\, j \notin J_{\Ip_2} \right\rangle,
			\end{array}\renewcommand{\arraystretch}{1.15}
		\end{equation}
		where again $\mon_1=\left\langle\left. \bb{x}^{\bb{m}_j} \,\right|\, j=1,...,p_{\mon_1}\right\rangle$ and 
		$\mon_2=\left\langle\left. \bb{x}^{\bb{u}_j} \,\right|\, j=1,...,p_{\mon_2}\right\rangle$. Recall 
		that $J_{\Ip_i}$ are all indices in $\left\{1,...,p_{\mon_i}\right\}$ whose variables were not set to zero by the removal of single 
		variables (Construction \ref{cons:removeSingleVariables}), and the $\bb{b}_j,\bb{c}_j$ are the exponents of $\bb{y}$ and $\bb{y}'$, respectively, of the remaining 
		binomials after Construction \ref{cons:universalFamilyOfComponent}. Note that we have already renumbered the coefficients, such that the new variables are 
		$\bb{y}=(y_1,...,y_r)$ and $\bb{y}'=(y_1',...,y_s')$ with $r=\rv(\Ip_1)$ and $s=\rv(\Ip_2)$. 
		
		\begin{lem}\label{lem:sameidentity}
			Let $J_1$ and $J_2$ be two universal families as in (\ref{eqn:twouniversalfamilies}) which parametrise the ambient torus of the same reduced non-coherent irreducible 
			component of $\HA$, then we have $\left.J_1\right._{\left(\bb{y}=(\bb{1})\right)}=\left.J_2\right._{\left(\bb{y}'=(\bb{1})\right)}$. This means the $\A$-graded 
			ideals given by the identity in $J_1$ and $J_2$ are the same.
		\end{lem}
		
		\begin{proof}
			If $J_1$ and $J_2$ parametrise the same $\A$-graded ideals then there exists some point $\bb{\lambda}=(\lambda_1,...,\lambda_s)$ such that 
			$\left.J_1\right._{\left(\bb{y}=(\bb{1})\right)}=\left.J_2\right._{\left(\bb{y}'=\bb{\lambda}\right)}$. Now let 
			$\bb{x}^{\bb{m}}-\bb{x}^{\bb{n}}$ be a Graver binomial in $\left.J_1\right._{\left(\bb{y}=(\bb{1})\right)}$. Because they are equal, 
			$\bb{x}^{\bb{m}}-\bb{x}^{\bb{n}} \in \left.J_2\right._{\left(\bb{y}'=\bb{\lambda}\right)}$ holds as well. But we have 
			$\bb{x}^{\bb{m}}-\bb{\lambda}^{\bb{w}}\cdot \bb{x}^{\bb{n}} \in \left.J_2\right._{\left(\bb{y}'=\bb{\lambda}\right)}$ for some $\bb{w} \in {\IZ}^{s}$, so we 
			must have $\bb{\lambda}^{\bb{w}}=1$ which is satisfied for $\bb{\lambda}=(\bb{1})$. This holds for every Graver binomial, so that 
			$\left.J_1\right._{\left(\bb{y}=(\bb{1})\right)}=\left.J_2\right._{\left(\bb{y}'=\bb{\lambda}\right)}=\left.J_2\right._{\left(\bb{y}'=(\bb{1})\right)}$ holds, since 
			by \cite{PeevaStillman:ToricHilbert} or \cite{Sturmfels:GeomAGraded} every {$\A$-graded} ideal is determined by its coefficients of the Graver binomials.
		\end{proof}
		
		\begin{rem}
			Note, that in fact $\bb{\lambda}=(\bb{1})$ since the parametrisation by $J_2$ is an isomorphism. Furthermore, if the two families $J_1$ and $J_2$ are isomorphic 
			then $r=s$ as the dimension of the component is the number of remaining variables in the universal family.
		\end{rem}
		
		On the other hand, consider two universal families $J_1$ and $J_2$ with equal ideals for the identity
		$\left.J_1\right._{\left(\bb{y}=(\bb{1})\right)}=\left.J_2\right._{\left(\bb{y}'=(\bb{1})\right)}$ and let 
		$\bb{x}^{\bb{m}_j}- \bb{x}^{\bb{n}_j}$ be in $\left.J_1\right._{\left(\bb{y}=(\bb{1})\right)}$ with $\bb{x}^{\bb{m}_j},\bb{x}^{\bb{n}_j}\notin J_i$. Since 
		the fibers of the universal families in the point $(\bb{1})$ are the same  there is a unique $\bb{b}_j' \in \IZ^s$ such that 
		$\bb{x}^{\bb{m}_j}- \bb{y}'^{\bb{b}_j'} \cdot \bb{x}^{\bb{n}_j} \in J_2$. In fact, this existence is not trivial so that we will construct $\bb{b}_j'$ 
		explicitly. For this we denote by 
		\[\bb{x}^{\bb{m}_j}- \bb{x}^{\bb{n}_j}=\sum_{i\in J_{\Ip_2}} p_i(\bb{x})\cdot \left(\bb{x}^{\bb{u}_i}- \bb{x}^{\bb{v}_i}\right)\]
		a decomposition into the generators given by $\left.J_2\right._{\left(\bb{y}'=(\bb{1})\right)}$ where the $p_i(\bb{x})$ are polynomials in $\bb{x}$. Then 
		we split the $p_i$ into monomials and rearrange the summands to get the telescoping series
		\[\bb{x}^{\bb{m}_j}- \bb{x}^{\bb{n}_j}=\sum_{i_k} m_{i_k}(\bb{x})\cdot \left(\bb{x}^{\bb{u}_{i_k}}- \bb{x}^{\bb{v}_{i_k}}\right),\]
		\textit{i.e.} $m_{i_k}(\bb{x})\bb{x}^{\bb{u}_{i_k}}-m_{i_{k-1}}(\bb{x})\bb{x}^{\bb{v}_{i_{k-1}}}=0$. We may assume that all $m_{i_k}$ are positive, 
		because otherwise we interchange $\bb{x}^{\bb{u}_{i_k}}$ and $\bb{x}^{\bb{v}_{i_k}}$.
		If we insert the appropriate terms from $\field[\bb{y}'^{\pm 1}]$ on the right hand side we get a telescoping series in $\field[\bb{x},\bb{y}'^{\pm 1}]$
		\begin{equation}\label{eq:telescopingSum}
		 \sum_{i_k} \left(\left(\prod_{\nu < k} \bb{y}'^{\bb{c}_{i_\nu}}\right)m_{i_k}(\bb{x})\cdot \left(\bb{x}^{\bb{u}_{i_k}}- \bb{y}'^{\bb{c}_{i_k}}\cdot\bb{x}^{\bb{v}_{i_k}}\right)\right),
		\end{equation}
		where we take the negative $\bb{y}'$ exponent if we interchanged the $\bb{x}$ terms. But then this telescoping sum equals 
		$\bb{x}^{\bb{m}_j}- \bb{y}'^{\bb{b}_j'} \cdot \bb{x}^{\bb{n}_j}$, where 
		\[\bb{b}_j'=\sum_\nu \bb{c}_{i_\nu}.\]
		Hence, the exponent $\bb{b}_j'$ is a linear combination of the $\bb{c}_i$. Because $\bb{x}^{\bb{m}_j},\bb{x}^{\bb{n}_j}\notin J_i$ we have that $\bb{b}_j'$ 
		is unique. This means when rearranging the generators of $J_2$ to get the same $\bb{x}$ binomials as in $J_1$ we get the new $\bb{y}'$ exponents $\bb{b}_j'$ 
		as linear combinations of the $\bb{c}_i$. Thus, we get the following proposition:
		
		\begin{prop}\label{prop:representationofuniversalfamilies}
			Let $J_1$ and $J_2$ be two universal families given by equation (\ref{eqn:twouniversalfamilies}) such that 
			$\left.J_1\right._{\left(\bb{y}=(\bb{1})\right)}=\left.J_2\right._{\left(\bb{y}'=(\bb{1})\right)}$. Set $n_1:=\#(J_{\Ip_1})$ and $n_2:=\#(J_{\Ip_2})$. Then 
			there is an $n_1 \times n_2$ matrix $B_{1,2}$ and an $n_2 \times n_1$ matrix $B_{2,1}$ such that for the binomials 
			$\bb{x}^{\bb{m}_j}- \bb{y}'^{\bb{b}_j'} \cdot \bb{x}^{\bb{n}_j}\in J_2$ and $\bb{x}^{\bb{u}_j}- \bb{y}^{\bb{c}_j'} \cdot \bb{x}^{\bb{v}_j}\in J_1$ we have
			\begin{eqnarray*}
				\left(\bb{b}_j'\right)_{j\in J_{\Ip_1}} & = & \left(\bb{c}_i\right)_{i\in J_{\Ip_2}} \cdot B_{2,1} \\
				\left(\bb{c}_j'\right)_{j\in J_{\Ip_2}} & = & \left(\bb{b}_i\right)_{i\in J_{\Ip_1}} \cdot B_{1,2}.
			\end{eqnarray*}\qed
		\end{prop}
		
		\begin{rem}
			If, on the other hand, we start with the $\bb{x}^{\bb{m}_j}- \bb{y}'^{\bb{b}_j'}\cdot \bb{x}^{\bb{n}_j}\in J_2$, then by the same argument as before we can 
			reconstruct the $\bb{x}^{\bb{u}_j}- \bb{y}'^{\bb{c}_j}\cdot \bb{x}^{\bb{v}_j}$, so that in fact
			\[J_2 = \left\langle\left. \bb{x}^{\bb{m}_j}- \bb{y}'^{\bb{b}_j'}\cdot \bb{x}^{\bb{n}_j}\,\right|\, j\in J_{\Ip_1} \right\rangle 
					+ \left\langle\left. \bb{x}^{\bb{m}_j} \,\right|\, j \notin J_{\Ip_1} \right\rangle.\]
		\end{rem}
		
		Using the above we can give a complete description of when two universal families give the same non-coherent component of the toric Hilbert scheme.
		
		\begin{thm}\label{thm:isomorphicuniversalfamilies}
			Two universal families $J_1$ and $J_2$ parametrise the ambient torus of the same reduced non-coherent component of the toric Hilbert scheme if and only if we have 
			$\left.J_1\right._{\left(\bb{y}=(\bb{1})\right)}=\left.J_2\right._{\left(\bb{y}'=(\bb{1})\right)}$, $r=s$ for $\bb{y}=(y_1,...,y_r)$ and 
			$\bb{y}'=(y_1',...,y_s')$, and there exists an isomorphism $\Phi \in \Gl(r,\IZ)$ such that $\Phi(\bb{b}_j)=\bb{b}_j'$ for $j\in J_{\Ip_1}$ in the notation 
			of Proposition \ref{prop:representationofuniversalfamilies}.
		\end{thm}
		
		\begin{proof}
			If $J_1$ and $J_2$ parametrise the ambient torus of the same non-coherent component, then by Corollary \ref{rem:dimensionofcomponent} their dimensions must be the 
			same, so that $r=s$. We also get $\left.J_1\right._{\left(\bb{y}=(\bb{1})\right)}=\left.J_2\right._{\left(\bb{y}'=(\bb{1})\right)}$ by 
			Lemma \ref{lem:sameidentity} so that we can write
			\begin{eqnarray*}
			J_1 & = &  \left\langle\left. \bb{x}^{\bb{m}_j}- \bb{y}^{\bb{b}_j}\cdot \bb{x}^{\bb{n}_j}\,\right|\, j\in J_{\Ip_1} \right\rangle 
					+ \left\langle\left. \bb{x}^{\bb{m}_j} \,\right|\, j \notin J_{\Ip_1} \right\rangle \textrm{ and} \\
			J_2 & = &  \left\langle\left. \bb{x}^{\bb{m}_j}- \bb{y}'^{\bb{b}_j'}\cdot \bb{x}^{\bb{n}_j}\,\right|\, j\in J_{\Ip_1} \right\rangle 
					+ \left\langle\left. \bb{x}^{\bb{m}_j} \,\right|\, j \notin J_{\Ip_1} \right\rangle.
		\end{eqnarray*}
		It remains to show the equivalence to the $\IZ$-isomorphism $\Phi$. But $J_1$ and $J_2$ parametrise 
		the ambient torus of the same non-coherent component exactly if there is an isomorphism of their parametrising tori $\Spec(\field[\bb{y}^{\pm 1}])$ and 
		$\Spec(\field[\bb{y}'^{\pm 1}])$ which gives exactly the same points on $\HA$ by the two universal families $J_1$ and $J_2$. This means, if and only if there exists some 
		$\Phi': \field[\bb{y}^{\pm 1}] \rightarrow \field[\bb{y}'^{\pm 1}]$ such that 
		\[\Phi'(J_1)=\left\langle\left. \bb{x}^{\bb{m}_j}- \Phi'\left(\bb{y}^{\bb{b}_j}\right)\cdot \bb{x}^{\bb{n}_j}\,\right|\, j\in J_{\Ip_1} \right\rangle 
					+ \left\langle\left. \bb{x}^{\bb{m}_j} \,\right|\, j \notin J_{\Ip_1} \right\rangle=J_2.\]
		This is equivalent to an isomorphism $\Phi \in \Gl(r,\IZ)$ such that $\Phi(\bb{b}_j)=\bb{b}_j'$.
		\end{proof}
		
		\begin{rem}
			Construction \ref{cons:universalFamilyOfComponent} and Theorem \ref{thm:isomorphicuniversalfamilies} allows us to compute all non-coherent components of a toric 
			Hilbert Scheme for a given $\A$. First, one has to compute for each monomial $\A$-graded ideal $\mon$ the universal families of all non-coherent components containing 
			$\mon$. Then one collects all isomorphic universal families into one component. By doing this, one has for each component already the list of all monomial $\A$-graded ideals contained 
			in this component.
			
			Of course these computations are best done by a computer. Thus, all the constructions and algorithms in this work are contained in {\sc ToricHilbertSchemes} \cite{ToricHilbertSchemes},
			a package for {\sc Macaulay2} \cite{M2}, and all examples have been computed with this package.
		\end{rem}
		
		\begin{ex}\label{ex:1347}
			Let $\A=\left\{1,3,4,7\right\}$ with $d=1$. Consider the universal family
			\[J_1:=J_\mon(\Ip)= \left\langle b^2-y_3a^2c,bd^2-y_7ac^4,d^4-y_9c^7,a^3,ab,bc,ad,a^2c^2,ac^5\right\rangle\]
			for the coherent $\mon = \left\langle a^3,ab,b^2,bc,ad,a^2c^2,bd^2,ac^5,d^4\right\rangle$ and the universal family
			\[J_2:=J_{\mon_0}(\Ip_0)=\left\langle b^2-y_3a^2c,ac^4-y_7bd^2,d^4-y_9c^7,a^3,ab,bc,ad,a^2c^2,bd^3\right\rangle\]
			for the non-coherent $\mon_0=\left\langle a^3,ab,b^2,bc,ad,a^2c^2,ac^4,bd^3,d^4\right\rangle$. When substituting $\left(\bb{1}\right)$ we get
			\[\left.J_1\right._{\left(\bb{y}=(\bb{1})\right)}=\left\langle b^2-a^2c,bd^2-ac^4,d^4-c^7,a^3,ab,bc,ad,a^2c^2,ac^5\right\rangle\]
			and
			\[\left.J_2\right._{\left(\bb{y}=(\bb{1})\right)}=\left\langle b^2-a^2c,ac^4-bd^2,d^4-c^7,a^3,ab,bc,ad,a^2c^2,bd^3\right\rangle.\]
			Furthermore, since we have both, $ac^5=c(ac^4-bd^2)+d\cdot bc\in \left.J_1\right._{\left(\bb{y}=(\bb{1})\right)}$ as well as
			$bd^3=d(bd^2-ac^4)+ c^4\cdot ad \in \left.J_2\right._{\left(\bb{y}=(\bb{1})\right)}$, we get 
			$\left.J_1\right._{\left(\bb{y}=(\bb{1})\right)}=\left.J_2\right._{\left(\bb{y}=(\bb{1})\right)}$.\\
			The number of remaining variables in $J_1$ and in $J_2$ are $3$, hence $r=s$. Finally, we have
			\[ \bb{b}_1=\left(\begin{smallmatrix}-1\\0\\0\end{smallmatrix}\right),\, \bb{b}_2=\left(\begin{smallmatrix}0\\-1\\0\end{smallmatrix}\right),\, 
					\bb{b}_3=\left(\begin{smallmatrix}0\\0\\-1\end{smallmatrix}\right),\]
			\[\bb{b}_1'=\left(\begin{smallmatrix}-1\\0\\0\end{smallmatrix}\right),\, \bb{b}_2'=\left(\begin{smallmatrix}0\\1\\0\end{smallmatrix}\right),\, 
					\bb{b}_3'=\left(\begin{smallmatrix}0\\0\\-1\end{smallmatrix}\right),\]
			so that 
			\[\Phi = \left(\begin{array}{ccc} 1&0&0\\0&-1&0\\0&0&1\end{array}\right)\]
			is an isomorphism of the two universal families. Hence, $J_1$ and $J_2$ describe the same non-coherent component and both $\mon$ and $\mon_0$ lie on that component. When applying 
			Construction \ref{cons:universalFamilyOfComponent} to the remaining $51$ $\A$-graded monomial ideals we get that only six of them are also contained in a non-coherent 
			component, which is in fact the same component:
			\begin{eqnarray*}
				\mon_1 & = & \left\langle a^3,ab,b^2,bc,ad,a^2c^2,ac^4,bd^3,c^7\right\rangle,\\
				\mon_2 & = & \left\langle a^3,ab,a^2c,bc,ad,b^3,b^2d,bd^2,ac^5,d^4\right\rangle,\\
      	\mon_3 & = & \left\langle a^3,ab,b^2,bc,ad,a^2c^2,bd^2,ac^5,c^7\right\rangle,\\
      	\mon_4 & = & \left\langle a^3,ab,a^2c,bc,ad,b^3,b^2d,ac^4,bd^3,d^4\right\rangle,\\
      	\mon_5 & = & \left\langle a^3,ab,a^2c,bc,ad,b^3,b^2d,ac^4,bd^3,c^7\right\rangle, \textrm{ and}\\
      	\mon_6 & = & \left\langle a^3,ab,a^2c,bc,ad,b^3,b^2d,bd^2,ac^5,c^7\right\rangle,
			\end{eqnarray*}
			where all of them but $\mon_6$ are coherent. Thus, the component $V_{\Ip}$ given by $J_\mon(\Ip)$ contains $8$ monomial $\A$-graded ideals.\EX
		\end{ex}
		
		In this example it is quite clear from the equality of the two ideals when $(\bb{1})$ had been substituted that the two universal families define the same 
		non-coherent component. One could assume now that having the same $\A$-graded ideal given by the identity may suffice for two universal families to define the same component. 
		This would mean that Lemma \ref{lem:sameidentity} would in fact be an if and only if statement. The following is a counterexample to this.
		
		\begin{ex}[continuing \textbf{\ref{ex:The2by6example}}]
			Consider the $\A$-graded monomial ideal
			\begin{eqnarray*}
				\mon & = & \left\langle e^2,be,ae,cd,a^2c^2,b^3d,ac^{12},b^3c^9,b^2c^{12},ab^3c^8,d^3f^3,ab^2c^{11},d^2ef^4,d^4f^2,\right.\\
				  &   & \left.b^6c^6,d^5f,d^6,a^3d^2f^4,b^9c^5,b^{12}c^3,a^8df^5,b^{15}c^2,b^{18}\right\rangle
			\end{eqnarray*}
			with $23$ generators. The defining ideal of $\UM$ after removing redundant variables is
			\begin{eqnarray*}
				I_\mon' & = & \left\langle y_{15}y_{18}-y_{20}y_{21},y_{12}y_{13}y_{15}-y_2y_{12},y_{11}y_{13}y_{15}-y_2y_{11},\right.\\
						 &	 & \left. y_{20}y_{21}^2y_{23}-y_{11}y_{15},y_{18}y_{20}^2y_{21}-y_{11}y_{15},\right.\\
						 &	 & \left. y_{13}y_{15}y_{20}y_{21}-y_2y_{20}y_{21},y_{11}y_{15}y_{20}y_{21}-y_{12}y_{13}\right\rangle
			\end{eqnarray*}
			in $\field[y_1,...,y_{23}]$ which has a primary decomposition into $12$ primary ideals. Two of them are
			\begin{eqnarray*}
				\Iq_1 & = & \left\langle y_{11},y_{12},y_{18}^2,y_{18}y_{21},y_{21}^2, y_{15}y_{18}-y_{20}y_{21},y_{13}y_{15}-y_2\right\rangle \textrm{ and}\\
				\Iq_2 & = & \left\langle y_{11},y_{12},y_{18},y_{21}\right\rangle,
			\end{eqnarray*}
			where $\Iq_1$ is not reduced. The prime ideal $\Ip_1=\left\langle y_{11},y_{12},y_{18},y_{21},y_2-y_{13}y_{15}\right\rangle$ is the radical of the primary ideal $\Iq_1$. 
			The two reduced components corresponding to these prime ideals are given by the universal families
			\begin{eqnarray*}
				J_1 & = & \left\langle be-y_1y_3ac^2,d^2ef^4-y_1b^5c^8,b^6c^6-y_3ad^2f^4,b^{12}c^3-y_4a^6df^5,\right. \\
						&		& \left. b^{18}-y_5a^{11}f^6\right\rangle + \left\langle e^2,ae,cd,a^2c^2,b^3d,ac^{12},b^3c^9,b^2c^{12},ab^3c^8,\right. \\
						&		& \left. d^3f^3,ab^2c^{11},d^4f^2,d^5f,d^6,a^3d^2f^4,b^9c^5,a^8df^5,b^{15}c^2\right\rangle
			\end{eqnarray*}
			and 
			\begin{eqnarray*}
				J_2 & = & \left\langle be-y_2ac^2,d^2ef^4-y_1b^5c^8,b^6c^6-y_3ad^2f^4,b^{12}c^3-y_4a^6df^5,\right. \\
						&		& \left. b^{18}-y_5a^{11}f^6\right\rangle + \left\langle e^2,ae,cd,a^2c^2,b^3d,ac^{12},b^3c^9,b^2c^{12},ab^3c^8,\right.\\
						&		& \left. d^3f^3,ab^2c^{11},d^4f^2,d^5f,d^6,a^3d^2f^4,b^9c^5,a^8df^5,b^{15}c^2\right\rangle,
			\end{eqnarray*}
			where we have mapped $y_{13},y_{15},y_{20},y_{23}$ to $y_1,y_3,y_4,y_5$ and have replaced $y_2$ by $y_1y_3$ in $J_1$.
			Not only are these two not isomorphic by construction, also the first one is four-dimensional and the second five-dimensional. But substituting 
			$\bb{y} =\left(\bb{1}\right)$ in $J_1$ and $J_2$ gives the same ideal
			\[\renewcommand{\arraystretch}{1.15}\begin{array}{l}\left\langle be-ac^2,d^2ef^4-b^5c^8,b^6c^6-ad^2f^4,b^{12}c^3-a^6df^5,b^{18}-a^{11}f^6\right\rangle + \\
				\left\langle e^2,ae,cd,a^2c^2,b^3d,ac^{12},b^3c^9,b^2c^{12},ab^3c^8,d^3f^3,\right.\\
				\left.ab^2c^{11},d^4f^2,d^5f,d^6,a^3d^2f^4,b^9c^5,a^8df^5,b^{15}c^2\right\rangle.\end{array}\renewcommand{\arraystretch}{1}\]
			Furthermore, as can be seen from $\Ip_1$ and $\Ip_2$, the reduced component $V_{\Ip_1}$ is an embedded component in $V_{\Ip_2}$.\EX		
		\end{ex}
		
\section{The Polytope}        %

		This section is about the construction of the polytope of a non-coherent component. Recall, that the coherent component is given by the state polytope of the toric 
		ideal $I_{\A}$. For a non-coherent component we will show that the polytope is again a state polytope. Unfortunately, $J_\mon(\Ip)$ and $J_\mon(\Ip)'$ are not 
		necessarily homogeneous with respect to a strictly positive grading. But the non-coherent component given by the universal family is the projective closure, so that 
		we want to homogenise the universal family. Thus, we define a last little modification we will be using to construct the polytope whose normal fan is the normalisation 
		of the reduced non-coherent component $V_{\Ip}$.
	  
	  \begin{df}\label{df:generalisedUniversalFamily}
	  	Let $\mon$ be a monomial $\A$-graded ideal, $\Ip$ a prime ideal as in Proposition \ref{prop:coherentprimaryideals} or 
	  	\ref{prop:incoherentprimaryideals}, and 
	  	\[J_\mon(\Ip)=\left\langle\left. \bb{x}^{\bb{m}_j}- \bb{y}^{\bb{b}_j}\cdot s_j\,\right|\, j\in J_{\Ip} \right\rangle + 
	  			\left\langle\left. \bb{x}^{\bb{m}_j} \,\right|\, j \notin J_{\Ip}\right\rangle\]
	  	in $\field\left[\left.\bb{x},y_i^{\pm 1} \,\right|\, i \in \rv(\Ip)\right]$ the universal family of the component 
	  	$V_{\Ip}$. Then we consider an additional set of variables $\{z_i \,|\, i \in \rv(\Ip)\}$ and define the \emph{generalised universal family of the 
	  	component $V_{\Ip}$} as
	  	\[\JMP=\left\langle\left. \bb{z}^{\bb{b}_j^+}\bb{y}^{\bb{b}_j^-}\cdot\bb{x}^{\bb{m}_j}- \bb{z}^{\bb{b}_j^-}\bb{y}^{\bb{b}_j^+}\cdot s_j
	  				\,\right|\, j\in J_{\Ip} \right\rangle + \left\langle\left. \bb{x}^{\bb{m}_j}\,\right|\, j \notin J_{\Ip}\right\rangle\]
	  	in $\field[\bb{x},y_i,z_i \,|\, i \in \rv(\Ip)]$.
	  \end{df}
	  
	  \begin{remnn}
	  	Note that $\JMP$ is just a homogenisation of $J_\mon(\Ip)'$.
	  \end{remnn}
		
		\begin{ex}[continuing \textrm{\ref{ex:The2by6example}}]
			The universal family was
			\begin{eqnarray*}
	  	 		J_\mon(\Ip) & = & \left\langle ae-y_3y_4b^2, bf^6-y_1c^6e^5, b^4ce^4-\frac{y_3y_4}{y_2}d^3f^3,\right.\\
	  	 						 &   & \left. ad^3f^3-y_2b^6ce^3, ab^9-y_3d^6, d^6e-y_4b^{11}\right\rangle + \\
	  	 						 &	 & \left\langle bc^2, ac^2, cd, abcf, a^2cf, b^3cf, a^2bf^2, a^3f^2,\right.\\
	  	 		    		 &   & \left. ab^3f^2, c^8e^3, b^5f^2, df^5, d^2f^4, af^6,d^4f^2, b^9c\right\rangle
	  	\end{eqnarray*}
	  	in $\field[a,b,c,d,e,f,y_1,y_2,y_3,y_4]$ so that the generalised universal family is
	  	 \begin{eqnarray*}
	  	 		\JMP & = & \left\langle z_3z_4ae-y_3y_4b^2, z_1bf^6-y_1c^6e^5, y_2z_3z_4b^4ce^4-y_3y_4z_2d^3f^3,\right.\\
	  	 												 &   & \left. z_2ad^3f^3-y_2b^6ce^3, z_3ab^9-y_3d^6, z_4d^6e-y_4b^{11} \right\rangle +\\
	  	 												 &	 & \left\langle bc^2, ac^2, cd, abcf, a^2cf, b^3cf, a^2bf^2, a^3f^2,\right.\\
	  	 												 &   & \left. ab^3f^2, c^8e^3, b^5f^2, df^5, d^2f^4, af^6,d^4f^2, b^9c\right\rangle
	  	 \end{eqnarray*}
	  	 in $\field[a,...,f,y_1,...,y_4,z_1,...,z_4]$.
	  \end{ex}
		
		We have done all the preliminary work now, so that we can start directly with the theorem and the rest of the section will construct the proof of it in four steps.	  
	  
	  \begin{thm}\label{thm:ThePolytope}
	  	Let $\mon$ be a monomial $\A$-graded ideal and consider a generalised universal family $\JMP \subseteq \field[\bb{x},y_i,z_i \,|\, i \in \rv(\Ip)]$ of 
	  	a reduced component $V_{\Ip}$ containing $\mon$. Then $\JMP$ is homogeneous with respect to a strictly positive grading and the 
	  	normalisation of the component $V_{\Ip}$ is the toric variety defined by the normal fan of the state polytope $\state(\JMP)$, \textit{i.e.} the Gröbner fan of $\JMP$. 
	  \end{thm}
	  
	  For more details on state polytopes see \cite[Chapters 1-3]{Sturmfels:Groebner}. Before we prove the theorem we have to show four steps we will use in the proof.
	  
	  \begin{lem}\label{lem:initialIdealSub}
	  	Let $\mon'\subset \field[\bb{x},\bb{y},\bb{z}]$ be an initial monomial ideal of $\JMP$ with respect to a term order on $\field[\bb{x},\bb{y},\bb{z}]$. 
	  	Then $\mon_1 := \mon'_{\left(\bb{y}=\bb{z}=(\bb{1})\right)}$ is a monomial $\A$-graded ideal in the component $V_{\Ip}$. Furthermore, $\mon'$ is the only initial monomial 
	  	ideal of $\JMP$ with $\mon_1 = \mon'_{\left(\bb{y}=\bb{z}=(\bb{1})\right)}$.
	  \end{lem}
	  
	  \begin{proof}
	  	Recall that the number of $\bb{x}$ variables is $n$. We denote by $n_{\bb{y}}$ and $n_{\bb{z}}$ the number of $\bb{y}$ and $\bb{z}$ variables, respectively. 
	  	Then the torus
	  	\[\mathcal{T} := \left(\field^*\right)^{n+n_{\bb{y}}+n_{\bb{z}}}\]
	  	acts on $\field[\bb{x},\bb{y},\bb{z}]$ coordinate-wise. Let $\lambda=\left(\lambda_1,...,\lambda_n,\lambda_{\bb{y}},\lambda_{\bb{z}}\right)\in \mathcal{T}$ 
	  	be arbitrary, where  $\lambda_{\bb{y}}$ and $\lambda_{\bb{z}}$ denote the coordinates acting respectively on $\bb{y}$ and $\bb{z}$. Note that $(\lambda_1,...,\lambda_n)$ 
	  	is in fact an element of the torus $T$ that acts coordinate-wise on $S=\field[x_1,...,x_n]$. By trivial extension to $\bb{y}$ and $\bb{z}$ such an element 
	  	$(\lambda_1,...,\lambda_n)$ also acts on $\field[\bb{x},\bb{y},\bb{z}]$. Then one can easily 
	  	check that
	  	\[\begin{array}{rcl}
	  				\left(\lambda . \JMP\right)_{\left(\bb{y}=\bb{z}=(\bb{1})\right)} & = & 
	  						\left(\left(\lambda_1,...,\lambda_n\right) . \JMP\right)_{\left(\bb{y}=\lambda_{\bb{y}},\bb{z}=\lambda_{\bb{z}}\right)}\\
	  																													& = & \left(\lambda_1,...,\lambda_n\right) . \left(\left(\JMP\right)_{\left(\bb{y}=\lambda_{\bb{y}},\bb{z}=\lambda_{\bb{z}}\right)}\right)
	  		\end{array}\]
	  	holds. Since $\left(\JMP\right)_{\left(\bb{y}=\lambda_{\bb{y}},\bb{z}=\lambda_{\bb{z}}\right)}$ is $\A$-graded and in $V_{\Ip}$ for every $\lambda\in \mathcal{T}$ 
	  	and the $n$-torus orbit of an $\A$-graded ideal in $V_{\Ip}$ is an $\A$-graded ideal in $V_{\Ip}$, we also have that
	  	\[\left(\lambda . \JMP\right)_{\left(\bb{y}=\bb{z}=(\bb{1})\right)}\]
	  	is an $\A$-graded ideal in $V_{\Ip}$ for every $\lambda\in\mathcal{T}$. Moreover, the monomial ideal $\mon'$ is given as the initial monomial ideal with respect to 
	  	a weight vector $\omega \in \IN^{n+n_{\bb{y}}+n_{\bb{z}}}$. Hence, by using \cite[Theorem 15.17]{Eisenbud:CommAlg} we get that
	  	\[\mon_1=\mon'_{\left(\bb{y}=\bb{z}=(\bb{1})\right)} = \left(\init_{\omega}\left(\JMP\right)\right)_{\left(\bb{y}=\bb{z}=(\bb{1})\right)}\]
	  	is an $\A$-graded ideal in $V_{\Ip}$ and since $\mon'$ is monomial, $\mon_1$ is too.
	  	
	  	For the second part, we fix the minimal set of generators $\left\{f_1,...,f_l\right\}$ of $\mon_1$ and 
	  	denote by $s_i$ the standard monomial in the degree of $f_i$. Since $\mon_1$ is a monomial $\A$-graded ideal in $V_{\Ip}$, we get another generalised 
	  	universal family 
	  	\[\widetilde{J_{\mon_1}(\Ip_1)} = \left\langle\left. \bb{z}^{\bb{c}_i^+}\bb{y}^{\bb{c}_i^-}f_i - \bb{z}^{\bb{c}_i^-}\bb{y}^{\bb{c}_i^+}s_i \,\right|\, i \in J_{\Ip_1}\right\rangle + 
	  						\left\langle\left. f_i \,\right|\, i \notin J_{\Ip_1}\right\rangle, \]
	  	which also gives $V_{\Ip}$. Thus $\widetilde{J_{\mon_1}(\Ip_1)}$ is isomorphic to $\JMP$, so that in fact
	  	\[\JMP = \left\langle\left. \bb{z}^{\bb{b}_i^+}\bb{y}^{\bb{b}_i^-}f_i - \bb{z}^{\bb{b}_i^-}\bb{y}^{\bb{b}_i^+}s_i \,\right|\, i \in J_{\Ip_1}\right\rangle + 
	  						\left\langle\left. f_i \,\right|\, i \notin J_{\Ip_1}\right\rangle \]
	  	holds after a suitable change of $\bb{y}$ and $\bb{z}$ coordinates in $\widetilde{J_{\mon_1}(\Ip_1)}$ as in Theorem \ref{thm:isomorphicuniversalfamilies}. Because 
	  	all $s_i$ are not in $\mon_1$ the 
	  	$\bb{z}^{\bb{b}_i^-}\bb{y}^{\bb{b}_i^+}s_i$ are also not in $\mon'$, so that we can choose a term order $\prec$ on $\field[\bb{x},\bb{y},\bb{z}]$ such that
	  	\[\bb{z}^{\bb{b}_i^-}\bb{y}^{\bb{b}_i^+}s_i \prec \bb{z}^{\bb{b}_i^+}\bb{y}^{\bb{b}_i^-}f_i.\]
	  	Then we claim that
	  	\[G_{\mon_1} := \left\{\left.\bb{z}^{\bb{b}_i^+}\bb{y}^{\bb{b}_i^-}f_i - \bb{z}^{\bb{b}_i^-}\bb{y}^{\bb{b}_i^+}s_i \,\right|\, i \in J_{\Ip_1}\right\} \cup 
	  									\left\{ f_i \,\left|\, i \notin J_{\Ip_1} \right.\right\}\]
	  	together with a possibly empty set of monomials in $\field[\bb{x},\bb{y},\bb{z}]$ is the reduced Gröbner basis of $\JMP$ with respect to the term order 
	  	$\prec$. To show this, we start with a pair of binomials 
	  	\[\bb{z}^{\bb{b}_i^+}\bb{y}^{\bb{b}_i^-}f_i - \bb{z}^{\bb{b}_i^-}\bb{y}^{\bb{b}_i^+}s_i,\;\bb{z}^{\bb{b}_j^+}\bb{y}^{\bb{b}_j^-}f_j - \bb{z}^{\bb{b}_j^-}\bb{y}^{\bb{b}_j^+}s_j 
	  			\;\textrm{ in }\, G_{\mon_1}\] 
	  	and compute their S-polynomial $\bb{z}^{\bb{b}^+}\bb{y}^{\bb{b}^-}\bb{x}^{\bb{m}} - \bb{z}^{\bb{b}^-}\bb{y}^{\bb{b}^+}\bb{x}^{\bb{n}}$. Then there are two 
	  	cases, either we have that $\bb{z}^{\bb{b}^+}\bb{y}^{\bb{b}^-}\bb{x}^{\bb{m}}$ and $\bb{z}^{\bb{b}^-}\bb{y}^{\bb{b}^+}\bb{x}^{\bb{n}}$ are not in $\JMP$ or 
	  	they both are. First assume they are not. Then by using Construction \ref{cons:localcoherentequations} if $\mon_1$ is coherent or Construction 
	  	\ref{cons:localincoherentequations} if $\mon_1$ is non-coherent, $\bb{x}^{\bb{m}}-\bb{x}^{\bb{n}}$ reduces to zero via the binomials $f_i-s_i$ for 
	  	$i \in J_{\Ip_1}$ because both monomials reduce to the standard monomial in their common degree. But then we can use the telescoping sum 
	  	(\ref{eq:telescopingSum}) again to reduce $\bb{z}^{\bb{b}^+}\bb{y}^{\bb{b}^-}\bb{x}^{\bb{m}} - \bb{z}^{\bb{b}^-}\bb{y}^{\bb{b}^+}\bb{x}^{\bb{n}}$ via the 
	  	$\bb{z}^{\bb{b}_i^+}\bb{y}^{\bb{b}_i^-}f_i - \bb{z}^{\bb{b}_i^-}\bb{y}^{\bb{b}_i^+}s_i$ in $G_{\mon_1}$ to zero because the exponents $\bb{b}_i$ satisfy 
	  	the local equations. 
	  	
	  	On the other hand, if $\bb{z}^{\bb{b}^+}\bb{y}^{\bb{b}^-}\bb{x}^{\bb{m}}$ and $\bb{z}^{\bb{b}^-}\bb{y}^{\bb{b}^+}\bb{x}^{\bb{n}}$ are both in $\JMP$ then 
	  	their difference reduces either to a monomial in $\field[\bb{x},\bb{y},\bb{z}]$ or to zero. Therefore, $G_{\mon_1}$ together with a possibly empty set of 
	  	monomials is a Gröbner basis for $\prec$. Furthermore, the set $G_{\mon_1}$ itself is reduced, because no $f_i$ divides any $s_j$ and hence no 
	  	$\bb{z}^{\bb{b}_i^+}\bb{y}^{\bb{b}_i^-}f_i$ divides any $\bb{z}^{\bb{b}_j^-}\bb{y}^{\bb{b}_j^+}s_j$. The additional monomials do also not reduce any 
	  	element of $G_{\mon_1}$, because on the one side reducing one of the binomials would result in $\bb{z}^{\bb{b}_i^-}\bb{y}^{\bb{b}_i^+}s_i$ being in 
	  	$\JMP$ which is a contradiction, and on the other side the remaining $f_i$ are minimal generators of $\mon_1$ so that there cannot be any 
	  	$\bb{x}^{\bb{u}}\in \JMP$ dividing $f_i$. Thus, $G_{\mon_1}\cup\left\{\textrm{monomials in }\field[\bb{x},\bb{y},\bb{z}]\right\}$ is the reduced Gröbner 
	  	basis of $\JMP$ with respect to the term order $\prec$. Therefore, $\mon'$ is the unique initial monomial ideal of $\JMP$ that gives $\mon_1$ when 
	  	substituting $1$, because the reduced Gröbner basis is given by $\mon_1$.
	  \end{proof}
	  
	  Now we know the correspondence between initial monomial ideals of $\JMP$ and monomial ideals in $V_{\Ip}$. Next, we are interested in the Gröbner cone of these initial 
	  ideals.
	  
	  \begin{prop}\label{lem:coneForInitial}
	  	Let $\mon'$ be an initial monomial ideal of $\JMP$ with monomial $\A$-graded ideal 
	  	$\mon_1=\mon'_{\left(\bb{y}=\bb{z}=(\bb{1})\right)}=\left\langle \bb{x}^{\bb{m}_1},...,\bb{x}^{\bb{m}_l}\right\rangle$. Then the cone of maximal 
	  	dimension in the Gröbner fan of $\JMP$ corresponding to $\mon'$ is 
	  	\[\sigma = \left\{ \omega\,\left|\, \langle \omega, v_i\rangle\geq 0,\, \bb{x}^{\bb{m}_i}\notin \JMP\right.\right\}\quad\quad \textrm{for}\]
	  	\[v_i = \left(\begin{smallmatrix} \bb{m}_i-\bb{n}_i\\ -\bb{b}_i\\ \bb{b}_i \end{smallmatrix}\right),\]
	  	where $\bb{x}^{\bb{n}_i}$ is the standard monomial in the degree of $\bb{x}^{\bb{m}_i}$ and $\bb{b}_i\in \IZ^{\# \rv(\Ip)}$ unique such that 
	  	$\bb{z}^{\bb{b}_i^+}\cdot\bb{y}^{\bb{b}_i^-}\cdot \bb{x}^{\bb{m}_i}-\bb{z}^{\bb{b}_i^-}\cdot\bb{y}^{\bb{b}_i^+}\cdot \bb{x}^{\bb{n}_i} \in \JMP$.
	  \end{prop}
	  
	  \begin{proof}
	  	Let $\prec$ be a term order on $\field[\bb{x},\bb{y},\bb{z}]$ such that $\mon'=\init_{\prec}\left(\JMP\right)$. Then as shown in the proof of Lemma \ref{lem:initialIdealSub} 
	  	\[\left\{\bb{z}^{\bb{b}_i^+}\cdot\bb{y}^{\bb{b}_i^-}\cdot \bb{x}^{\bb{m}_i}-\bb{z}^{\bb{b}_i^-}\cdot\bb{y}^{\bb{b}_i^+}\cdot \bb{x}^{\bb{n}_i}
	  			\,\left|\, \bb{x}^{\bb{m}_i}\notin \JMP\right.\right\} \cup \left\{\bb{x}^{\bb{m}_i} \in \JMP\right\}\]
	  	is a subset of the reduced Gröbner basis of $\JMP$ with respect to $\prec$ that contains all binomials of the reduced Gröbner basis. Thus, the relative interior of the 
	  	corresponding Gröbner cone $\sigma$ is given by all 
	  	$\omega \in \IZ^{n+n_{\bb{y}}+n_{\bb{z}}}$ such that
	  	\[\init_{\omega}\left(\bb{z}^{\bb{b}_i^+}\cdot\bb{y}^{\bb{b}_i^-}\cdot \bb{x}^{\bb{m}_i}-\bb{z}^{\bb{b}_i^-}\cdot\bb{y}^{\bb{b}_i^+}\cdot \bb{x}^{\bb{n}_i}\right) = 
	  			\bb{z}^{\bb{b}_i^+}\cdot\bb{y}^{\bb{b}_i^-}\cdot \bb{x}^{\bb{m}_i}\]
	  	for all $\bb{x}^{\bb{m}_i}\notin \JMP$. But this holds exactly if
	  	\[\sigma = \left\{ \omega\,\left|\, \langle \omega, v_i\rangle\geq 0,\, \bb{x}^{\bb{m}_i}\notin \JMP\right.\right\}\quad\quad \textrm{for}\]
	  	\[v_i = \left(\begin{smallmatrix} \bb{m}_i-\bb{n}_i\\ -\bb{b}_i\\ \bb{b}_i \end{smallmatrix}\right).\]
	  \end{proof}
	  
	  Before we give the last lemma needed for the proof of the theorem we state a short proposition we need for this lemma.
	  
	  \begin{prop}\label{prop:JMPcong}
	  	Let 
	  	\[\JMP=\left\langle\left. \bb{z}^{\bb{b}_i^+}\bb{y}^{\bb{b}_i^-}\cdot\bb{x}^{\bb{m}_i}- \bb{z}^{\bb{b}_i^-}\bb{y}^{\bb{b}_i^+}\cdot \bb{x}^{\bb{n}_i}
	  				\,\right|\, i\in J_{\Ip} \right\rangle + \left\langle\left. \bb{x}^{\bb{m}_i}\,\right|\, i \notin J_{\Ip}\right\rangle\]
	  	be the generalised universal family of an irreducible component $V_{\Ip}$ containing $\mon$. Then we have
	  	\[\field[\bb{y}^{\bb{b}_i}\,|\, i\in J_{\Ip}, \bb{y}=(y_j)_{j\in \rv(\Ip)}]\cong 
	  			\field[\bb{z}^{\bb{b}_i}\bb{y}^{-\bb{b}_i}\bb{x}^{\bb{m}_i-\bb{n}_i}\,|\, i\in J_{\Ip}, \bb{y}=(y_j)_{j\in \rv(\Ip)}].\] 
	  \end{prop}
	  
	  \begin{proof}
	  	The first step is that for $\bb{y}=\left(y_j\right)_{j\in \rv(\Ip)}$ and $\bb{z}=\left(z_j\right)_{j\in \rv(\Ip)}$
	  	\[\field[\bb{y}^{\bb{b}_i}\,|\, i\in J_{\Ip}]\cong \field[\bb{z}^{\bb{b}_i}\bb{y}^{-\bb{b}_i}\,|\, i\in J_{\Ip}]\]
	  	holds, since this is just the diagonal embedding. Secondly, 
	  	\[\field[\bb{z}^{\bb{b}_i}\bb{y}^{-\bb{b}_i}\,|\, i\in J_{\Ip}] \cong \field[\bb{z}^{\bb{b}_i}\bb{y}^{-\bb{b}_i}\bb{x}^{\bb{m}_i-\bb{n}_i}\,|\, i\in J_{\Ip}]\]
	  	holds, because the relations on the $\bb{x}^{\bb{m}_i-\bb{n}_i}$ are the same as the relations on the $\bb{y}^{\bb{b}_i}$. In fact, any relation between the 
	  	$\bb{x}^{\bb{m}_i-\bb{n}_i}$ can be interpreted as the reduction of an S-polynomial in two of them. But this is exactly the construction of the local equations 
	  	which have been used for the base change to the $\bb{y}^{\bb{b}_i}$, so that these also satisfy these relations.
	 	\end{proof}
	 	
	 	\begin{lem}\label{lem:localAffineChart}
	 		Let $V_{\Ip}$ be a reduced irreducible component containing an $\A$-graded monomial ideal $\mon$ and let $\left\{\bb{b}_i\,|\,i\in J_{\Ip}\right\}$ be the exponent vectors in the 
	 		generalised universal family obtained from $\mon$ for this component. Then
	 		\[\left(\UM\right)_{\textrm{red}} \cap V_{\Ip} = \Spec(\field[\bb{z}^{\bb{b}_i}\bb{y}^{-\bb{b}_i}\bb{x}^{\bb{m}_i-\bb{n}_i}\,|\, i\in J_{\Ip}, \bb{y}=(y_j)_{j\in \rv(\Ip)}])\]
	 		holds.
	 	\end{lem}
	 	
	 	\begin{proof}
	 		To prove this we will go through the construction of the $\bb{b}_i$ exponents of $\bb{y}$ and $\bb{z}$ again. The affine chart of the 
	 		reduced structure of the irreducible component determined by $\Ip$ that contains $\mon$ is given by 
			\[\left(\UM\right)_{\textrm{red}}\cap V_{\Ip} \cong \Spec\left(\field[y_i\,|\, i \in \rv']/\Ip'\right),\]
			by Remark \ref{rem:localaffinechart}. Now let $A_1,...,A_h$ be the matrices used for the reduction as in Construction \ref{cons:Isomorphism} to remove the binomials 
			in $\Ip'$, let $A:=A_h\cdot ... \cdot A_1$ be their product, and we set $\tilde{\bb{b}_i}$ to be the $i$-th column of $A$. Thus, we get the surjective 
			morphism
			\[\begin{array}{ccccc}
					\Phi_A : &  \field[y_i\,|\, i \in \rv'] & \rightarrow & \field[\bb{y}^{\tilde{\bb{b}_i}} \,|\, i \in \rv', \bb{y}=(y_j)_{j\in \rv'}]/_{\left\langle y_k-1\right\rangle}&\\
									 &  y_i													& \mapsto & \bb{y}^{\tilde{\bb{b}_i}}&
				\end{array}\]
			for $k\in \rv'\setminus \rv(\Ip)$. By construction of $A$ the kernel of $\Phi_A$ is exactly $\Ip'$ so that we have
			\[\field[y_i\,|\, i \in \rv']/\Ip' \cong \field[\bb{y}^{\tilde{\bb{b}_i}} \,|\, i \in \rv', \bb{y}=(y_j)_{j\in \rv'}]/_{\left\langle y_k-1\right\rangle}.\]
			Now we set $y_k$ to $1$ for $k\in \rv'\setminus \rv(\Ip)$ by projecting the $\tilde{\bb{b}_i}$ to the $\rv(\Ip)$ variables, \textit{i.e.} if $\pi$ is the 
			projection from the $\rv'$ variables to the $\rv(\Ip)$ variables, then with	$\bb{b}_i:=\pi(\tilde{\bb{b}_i})$ we have
			\[\field[y_i\,|\, i \in \rv']/\Ip' \cong \field[\bb{y}^{\bb{b}_i} \,|\, i \in \rv', \bb{y}=(y_j)_{j\in \rv(\Ip)}],\]
			where the $\bb{b}_i$ are precisely as defined before. Note that the indices of the removed redundant variables are $J_{\Ip}\setminus \rv'$. But 
			for $i\in J_{\Ip}\setminus \rv'$ the resulting exponent of $J_\mon$ is $\bb{b}_i=\sum_{j\in \rv'} \lambda_j \bb{b}_j$ for $\lambda_j \in \IN$, 
			because $y_i$ was a redundant variable. Thus, adding $\bb{y}^{\bb{b}_i}$ for $i\in J_{\Ip}\setminus \rv'$ to the generators of the ring does not 
			change the ring, so that
			\[\field[y_i\,|\, i \in \rv']/\Ip' \cong \field[\bb{y}^{\bb{b}_i} \,|\, i \in J_{\Ip}, \bb{y}=(y_j)_{j\in \rv(\Ip)}].\]
			Finally, by using Proposition \ref{prop:JMPcong} we get
			\[\field[\bb{y}_i\,|\, i \in \rv']/\Ip' \cong \field[\bb{z}^{\bb{b}_i}\cdot\bb{y}^{-\bb{b}_i}\cdot\bb{x}^{\bb{m}_i-\bb{n}_i} \,|\, i \in J_{\Ip}, \bb{y}=(y_j)_{j\in \rv(\Ip)}].\]
	 	\end{proof}
	 	
	 	Note that for the coherent component this is similar to the construction in the proof of \cite[Theorem 4.1]{StillmanSturmfelsThomas:Algorithms}.
	 	
	 	Now we have collected all steps and can prove the theorem.
	  
	  \begin{proof}[Proof of Theorem \ref{thm:ThePolytope}]
	  	First of all, note that 
	  	\[\JMP=\left\langle\left. \bb{z}^{\bb{b}_j^+}\cdot\bb{y}^{\bb{b}_j^-}\cdot \bb{x}^{\bb{m}_j}-
	  		 \bb{z}^{\bb{b}_j^-}\cdot\bb{y}^{\bb{b}_j^+}\cdot \bb{x}^{\bb{n}_j}\,\right|\, j\in J_{\Ip} \right\rangle 
					+ \left\langle\left. \bb{x}^{\bb{m}_j} \,\right|\, j \notin J_{\Ip} \right\rangle\]
			is homogeneous with respect to a strictly positive grading for the degree vector $(\bb{a},\bb{1},\bb{1})$, where $\bb{a}$ is a strictly positive vector in 
			the row span of $\A$ and $\bb{1}$ is the classical degree vector on $\bb{y}$ and $\bb{z}$. Thus, by \cite[Theorem 2.5]{Sturmfels:Groebner} there exists a state polytope 
			\[P := \state(\JMP).\] 
			Now let $\sigma$ be a maximal cone in the normal fan of $P$, \textit{i.e.} in the Gröbner fan. This gives an initial monomial ideal 
			$\mon'\subseteq \field[\bb{x},\bb{y},\bb{z}]$ of $\JMP$, which in turn by Lemma \ref{lem:initialIdealSub} when substituting $\bb{y}=\bb{z}=(\bb{1})$ gives a 
			monomial $\A$-graded ideal $\mon_1$ in $V_{\Ip}$. There are two cases, either $\mon_1$ is the original monomial ideal $\mon$ we used to construct $\JMP$ or 
			$\mon_1$ is some other monomial ideal on that component.
			
			For the first case, assume this monomial is $\mon$. Then by Proposition \ref{lem:coneForInitial} the cone $\sigma$ is given by 
			$\{\omega\,|\, \langle \omega,v_i\rangle \geq 0\}$ for 
			\[v_i:= \left(\begin{smallmatrix}\bb{m}_i-\bb{n}_i\\ -\bb{b}_i\\ \bb{b}_i \end{smallmatrix}\right), i\in J_{\Ip}.\]
			Hence, we have that $\sigma^{\vee}$ is the positive hull over $\IQ$ of $\left\{ v_i \,|\, i \in J_{\Ip}\right\}$. The affine chart of $V_{\Ip}$, that contains $\mon$, is given by 
			\[\left(\UM\right)_{\textrm{red}} \cap V_{\Ip} \cong \Spec\left(\field[y_i\,|\, i \in \rv']/\Ip'\right),\]
			see Remark \ref{rem:localaffinechart}. But by Lemma \ref{lem:localAffineChart} we have 
			\[\left(\UM\right)_{\textrm{red}} \cap V_{\Ip} \cong \Spec(\field[\bb{z}^{\bb{b}_i}\bb{y}^{-\bb{b}_i}\bb{x}^{\bb{m}_i-\bb{n}_i}\,|\, i\in J_{\Ip}, \bb{y}=(y_j)_{j\in \rv(\Ip)}])\]
			and the exponent vectors on the right hand side are the $v_i$, the generators of the cone $\sigma^{\vee}$. Thus, if we denote by 
			$M_{\sigma}:=\left(\IQ\cdot\sigma^{\vee}\right) \cap \IZ^{n+n_{\bb{y}}+n_{\bb{z}}}$ the 
			lattice of $\sigma^{\vee}$, we conclude that
			\[\Spec\left(\field[\sigma^{\vee}\cap M_{\sigma}]\right)\]
			is the normalisation of 
			\[\Spec\left(\field[y_i\,|\, i \in \rv']/\Ip'\right),\]
			the affine chart of $V_{\Ip}$ containing $\mon$.
			
			Secondly, let $\mon_1\neq \mon$. Then for the monomial ideal $\mon_1$ we 
			get another universal family 
			\[J_{\mon_1}=\left\langle\left. \bb{x}^{\bb{u}_j}- \bb{y}'^{\bb{c}_j}\cdot \bb{x}^{\bb{v}_j}\,\right|\, j\in J_{\Ip_1} \right\rangle 
					+ \left\langle\left. \bb{x}^{\bb{u}_j} \,\right|\, j \notin J_{\Ip_1} \right\rangle\]
			where $\Ip_1$ is the prime ideal that gives $V_{\Ip}$ for $\mon_1$. By Theorem \ref{thm:isomorphicuniversalfamilies} there is an isomorphism 
			$\Phi : \field[\bb{y}'^{\pm 1}] \rightarrow \field[\bb{y}^{\pm 1}]$, such that $\Phi(J_{\mon_1})=J_\mon$. Thus, if we apply $\Phi$ (extended to $\bb{z}'$) to the general universal 
			family we get
			\begin{eqnarray*}
				\Phi(\widetilde{J_{\mon_1}(\Ip_1)}) & = & \left\langle\left. \bb{z}^{\Phi(\bb{c}_j)^+}\cdot\bb{y}^{\Phi(\bb{c}_j)^-}\cdot \bb{x}^{\bb{u}_j}-
	  		 \bb{z}^{\Phi(\bb{c}_j)^-}\cdot\bb{y}^{\Phi(\bb{c}_j)^+}\cdot \bb{x}^{\bb{v}_j}\,\right|\, j\in J_{\Ip_1} \right\rangle \\
					& & + \left\langle\left. \bb{x}^{\bb{u}_j} \,\right|\, j \notin J_{\Ip_1} \right\rangle\\
					& = & \JMP.
			\end{eqnarray*}
			Then Proposition \ref{lem:coneForInitial} implies, as it did for $\mon$, that $\sigma^{\vee}$ is generated by
			\[v_i':= \left(\begin{smallmatrix} \bb{u}_i-\bb{v}_i\\ -\Phi(\bb{c}_i)\\ \Phi(\bb{c}_i)\end{smallmatrix}\right), i\in J_{\Ip_1}.\]
			On the other hand we have $\left(\mathcal{U}_{\mon_1}\right)_{\textrm{red}}\cap V_{\Ip} \cong \Spec\left(\field[\bb{y}']/\Ip_1'\right)$ by Remark 
			\ref{rem:localaffinechart} and by the same argumentation as above we get
			\[\field[\bb{y}']/\Ip_1' \cong \field[\bb{z}'^{\bb{c}_i}\cdot\bb{y}'^{-\bb{c}_i}\cdot\bb{x}^{\bb{u}_i-\bb{v}_i} \,|\, i \in J_{\Ip_1}, \bb{y}=(y_j)_{j\in \rv(\Ip_1)}].\]
			But $\Phi$ is an isomorphism over $\IZ$ so if we apply $\Id_{\field[\bb{x}]}\otimes \Phi$ (again extended to $\bb{z}'$) to the right hand side we get
			\[\field[\bb{z}'^{\Phi(\bb{c}_i)}\cdot\bb{y}'^{-\Phi(\bb{c}_i)}\cdot\bb{x}^{\bb{u}_i-\bb{v}_i} \,|\, i \in J_{\Ip_1}, \bb{y}=(y_j)_{j\in \rv(\Ip_1)}],\]
			which implies that $\Spec\left(\field[\sigma^{\vee}\cap M_{\sigma}]\right)$ is the normalisation of the affine chart of $V_{\Ip}$ containing $\mon_1$. Note that for every 
			maximal cone $\sigma$ we get the same lattice $M_{\sigma}=:M_{\Ip}$.
			
			The normalisation maps for all maximal cones which we have just constructed each sends the identity point to the ideal 
			\[\left.J_\mon\right._{\left(\bb{y}=(\bb{1})\right)}=\left\langle\left. \bb{x}^{\bb{m}_j}- \bb{x}^{\bb{n}_j}\,\right|\, j\in J_{\Ip} \right\rangle 
					+ \left\langle\left. \bb{x}^{\bb{m}_j} \,\right|\, j \notin J_{\Ip} \right\rangle \subset V_{\Ip}.\] 
			Clearly, all these normalisation maps are equivariant under the 
			action of the torus $\Hom(M_{\Ip},\field^*)$. Hence, there exists a unique $\Hom(M_{\Ip},\field^*)$-equivariant morphism $\Psi$ from the projective toric variety given by the 
			Gröbner fan of $\JMP$ onto the non-coherent component $V_{\Ip}\subseteq \left(\HA\right)_{\textrm{red}}$, that restricts to the normalisation 
			maps constructed above on each affine open chart. Hence, $\Psi$ is the normalisation morphism from the projective toric variety, given by the normal fan of 
			the polytope $\state(\JMP)$, to $V_{\Ip}$.
		\end{proof}
		
		\begin{df}
			We call a state polytope of a generalised universal family, $\state\left(\JMP\right)$, a \emph{generalised state polytope of $\A$}.
		\end{df}
		
		\begin{ex}[continuing \textbf{\ref{ex:1347}}]
			For $\A=\left\{1,3,4,7\right\}$ and the monomial $\A$-graded ideal $\mon = \left\langle a^3,ab,b^2,bc,ad,a^2c^2,bd^2,ac^5,d^4\right\rangle$ we had the 
			universal family $J_\mon(\Ip)$ which is in fact the only one for $\mon$. Thus, $\mon$ is contained in one non-coherent component $V$ and 
			\[\JMP = \left\langle z_1b^2-y_1a^2c,z_2bd^2-y_2ac^4,z_3d^4-y_3c^7,a^3,ab,bc,ad,a^2c^2,ac^5\right\rangle\]
			is the generalised universal family of this reduced non-coherent irreducible component $V$. This component $V$ contains the eight monomial ideals 
			$\mon,\mon_0,\mon_1,...,\mon_6$. A state polytope for $\JMP$ is a cube in $\IQ^{10}$ (there are $10$ variables in total) with the 
			vertices corresponding to the monomial ideals in the following way:
			\[\begin{array}{rcl}
				\mon   & \leftrightarrow & \left(1,-1,4,-2,0,-1,0,0,1,0\right)^t   \\
				\mon_0 & \leftrightarrow & \left(0,0,0,0,0,0,0,0,0,0\right)^t			 \\
				\mon_1 & \leftrightarrow & \left(0,0,-7,4,0,0,1,0,0,-1\right)^t		 \\
				\mon_2 & \leftrightarrow & \left(-1,1,3,-2,-1,-1,0,1,1,0\right)^t  \\
				\mon_3 & \leftrightarrow & \left(1,-1,-3,2,0,-1,1,0,1,-1\right)^t  \\
				\mon_4 & \leftrightarrow & \left(-2,2,-1,0,-1,0,0,1,0,0\right)^t   \\
				\mon_5 & \leftrightarrow & \left(-2,2,-8,4,-1,0,1,1,0,-1\right)^t  \\
				\mon_6 & \leftrightarrow & \left(-1,1,-4,2,-1,-1,1,1,1,-1\right)^t 
			\end{array}\]
			A sketch of the polytope in its affine hull is given in Figure \ref{fig:GenstateOf1347}.
			\begin{figure}[ht]
				\begin{center}
					\psset{unit=0.5cm}
					\begin{pspicture}(0,0)(20,20)
						\psline(8,6)(18,6)
						\psline(8,6)(2,2)
						\psline(8,6)(8,16)
						\psline[linecolor=white, linewidth=1](2,12)(12,12)
						\psline[linecolor=white, linewidth=1](12,2)(12,12)
						\pspolygon[linestyle=solid](2,2)(12,2)(18,6)(18,16)(8,16)(2,12)
						\psline(2,12)(12,12)
						\psline(12,2)(12,12)
						\psline(12,12)(18,16)
						\uput{0.01}[315](12,12){$\mon_2$}
						\uput{0.01}[225](2,2){$\mon_0$}
						\uput{0.01}[315](8,6){$\mon_1$}
						\uput{0.01}[315](12,2){$\mon$}
						\uput{0.01}[0](18,6){$\mon_3$}
						\uput{0.01}[135](2,12){$\mon_4$}
						\uput{0.01}[90](8,16){$\mon_5$}
						\uput{0.01}[45](18,16){$\mon_6$}
					\end{pspicture}
				\end{center}
				\caption{state$(\JMP)$}\label{fig:GenstateOf1347}
			\end{figure}
			Note that the two non-coherent vertices $\mon_0$ and $\mon_6$ are on opposing sides of the polytope, thus the intersection with the coherent component is not given by a face of this 
			polytope. The polytope of the coherent component is three-dimensional and has $51$ vertices (the coherent monomial ideals). One can compute that the intersection is not a face of this 
			polytope either. See Figure \ref{fig:2facesof1347} for a sketch of the two-dimensional faces of the state polytope of $I_{\A}$ that contain the coherent monomial ideals of the 
			non-coherent component.
			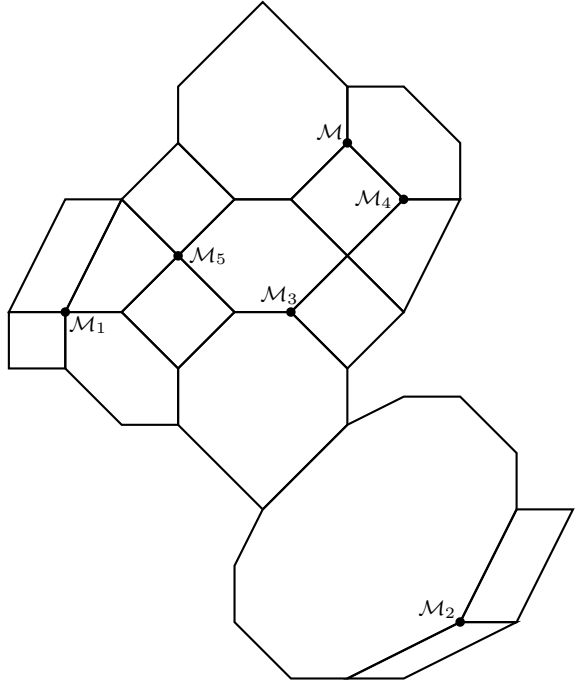
\begin{figure}[ht]
				\begin{center}
					\psset{unit=0.75cm}
					\begin{pspicture}(-1,-5.75)(11,6.75)
						\pspolygon[linestyle=solid](0,0)(1,0)(1,1)(0,1)
						\pspolygon[linestyle=solid](0,1)(1,1)(2,3)(1,3)
						\pspolygon[linestyle=solid](1,1)(2,1)(3,2)(2,3)
						\pspolygon[linestyle=solid](1,1)(2,1)(3,0)(3,-1)(2,-1)(1,0)
						\pspolygon[linestyle=solid](2,1)(3,2)(4,1)(3,0)
						\pspolygon[linestyle=solid](4,1)(5,1)(6,0)(6,-1)(4.5,-2.5)(3,-1)(3,0)
						\pspolygon[linestyle=solid](2,3)(3,2)(4,3)(3,4)
						\pspolygon[linestyle=solid](3,2)(4,1)(5,1)(6,2)(5,3)(4,3)
						\pspolygon[linestyle=solid](5,1)(6,0)(7,1)(6,2)
						\pspolygon[linestyle=solid](6,2)(7,1)(8,3)(7,3)
						\pspolygon[linestyle=solid](5,3)(6,2)(7,3)(6,4)
						\pspolygon[linestyle=solid](7,3)(6,4)(6,5)(7,5)(8,4)(8,3)
						\pspolygon[linestyle=solid](4,3)(5,3)(6,4)(6,5)(4.5,6.5)(3,5)(3,4)
						\pspolygon[linestyle=solid](6,-1)(7,-0.5)(8,-0.5)(9,-1.5)(9,-2.5)(8,-4.5)(6,-5.5)(5,-5.5)(4,-4.5)(4,-3.5)(4.5,-2.5)
						\pspolygon[linestyle=solid](9,-2.5)(10,-2.5)(9,-4.5)(8,-4.5)
						\pspolygon[linestyle=solid](8,-4.5)(6,-5.5)(7,-5.5)(9,-4.5)
						\psdots(1,1)(3,2)(5,1)(7,3)(6,4)(8,-4.5)
						\uput{0.1}[315](1,1){\scriptsize{$\mon_1$}}
						\uput{0.2}[0](3,2){\scriptsize{$\mon_5$}}
						\uput{0.15}[113](5,1){\scriptsize{$\mon_3$}}
						\uput{0.2}[180](7,3){\scriptsize{$\mon_4$}}
						\uput{0.1}[135](6,4){\scriptsize{$\mon$}}
						\uput{0.1}[135](8,-4.5){\scriptsize{$\mon_2$}}
					\end{pspicture}
				\end{center}
				\caption{The two-dimensional faces of $\state(I_{\A})$ that contain the coherent monomial ideals of the non-coherent component}\label{fig:2facesof1347}
			\end{figure}
			
			Furthermore, Sturmfels already computed a family of $\A$-graded ideals for this $\A$ in the proof of \cite[Theorem 10.4]{Sturmfels:Groebner}:
			\begin{equation}\label{eq:Sturmfels}
				\begin{array}{l}\left\langle x_1^2x_3-c_1x_2^2,x_1x_3^4-c_2x_2x_4^2,x_3^7-c_3x_4^4,\right.\\
												\left. x_1^3,x_1x_2,x_1x_4,x_2^3,x_2^2x_4,x_2x_3,x_2x_4^3\right\rangle\end{array}
			\end{equation}
			for $c_1,c_2,c_3\in \field^*$. Moreover, he showed that there is a $\field^*$ parametrisation of this family identifying all $\A$-graded ideals with $c_1c_3/c_2^2$ constant, 
			because these are isomorphic as $\A$-graded ideals. The family (\ref{eq:Sturmfels}) is exactly the ambient torus of the non-coherent component constructed above if we identify 
			$x_1,x_2,x_3,x_4$ with $a,b,c,d$, and $c_1,c_2,c_3$ with $\frac{z_1}{y_1},\frac{z_2}{y_2},\frac{z_3}{y_3}$, respectively. Computing the primary decomposition 
			of the defining ideal of $\mon$ gives
			\[I_\mon' = \left\langle y_6,y_8\right\rangle \cap \left\langle y_7^2-y_3y_9, y_6y_7-y_8y_9,y_3y_6-y_7y_8\right\rangle.\]
			The first primary ideal yields the non-coherent component $V$ and the second ideal the coherent component. Hence,
			\[\left\langle y_7^2-y_3y_9 \right\rangle\]
			gives the intersection of $V$ with the coherent component. Note, that we have identified $y_3$ with $\frac{y_1}{z_1}$, $y_7$ with $\frac{y_2}{z_2}$, and 
			$y_9$ with $\frac{y_3}{z_3}$. Thus, the isomorphism class $c_1c_3/c_2^2=1$ corresponds exactly to the intersection of $V$ with the coherent component.\EX
		\end{ex}
		
	  \begin{rem}
			This example implies that two $\A$-graded ideals, that correspond to points in the same orbit in a non-coherent component, need not be isomorphic as $\A$-graded ideals. This is in 
			contrast to the coherent component, where the orbits are exactly the isomorphism classes.
		\end{rem}

\phantomsection
\addcontentsline{toc}{chapter}{References}
\bibliography{Inco}
\address{
Ren\'e Birkner\\
Mathematisches Institut\\
Freie Universit\"at Berlin\\
Arnimallee 3\\
14195 Berlin, Germany}{rbirkner@math.fu-berlin.de}

\end{document}